\documentclass[a4paper,11pt]{article}

\usepackage{enumerate}
\usepackage{caption}

\usepackage{amsmath}
\usepackage{amsthm}
\usepackage{mathrsfs}
\usepackage{epsfig}
\usepackage{verbatim}
\usepackage[T1]{fontenc}
\usepackage{amsfonts}
\usepackage{float}
\usepackage{caption}
\usepackage{graphicx}
\usepackage{color}
\usepackage{fullpage}

\usepackage[latin2]{inputenc}
\usepackage{url}

\newtheorem{theorem}{Theorem}[section]

\newtheorem{proposition}[theorem]{Proposition}
\newtheorem{lemma}[theorem]{Lemma}
\newtheorem{corollary}[theorem]{Corollary}
\newtheorem{definition}[theorem]{Definition}

\newtheorem{assumptions}[theorem]{Assumptions}

\newtheorem{remark}[theorem]{Remark}

\makeindex

\begin{document}
\title{Transfer of energy in Camassa-Holm and related models by use of nonunique characteristics \\
\small{}}
\author{{Grzegorz Jamr\'oz} \\
{\it \small Institute of Mathematics, University of Basel, Spiegelgasse 1,
4051 Basel, Switzerland}\\ {\it \small  Institute of Mathematics, Polish Academy of Sciences, ul. \'Sniadeckich 8, 00-656 Warszawa, Poland}\\
{\it \small e-mail: jamroz@impan.pl}
}

\maketitle

\abstract{We study the propagation of energy density in finite-energy weak solutions of the Camassa-Holm and related equations. Developing the methods based on generalized nonunique characteristics, we show that the parts of energy related to positive and negative slopes are one-sided weakly continuous and of bounded variation, which allows us to define certain measures of dissipation of both parts of energy. The result is a step towards the open problem of uniqueness of dissipative solutions of the Camassa-Holm equation.}
{\bf MSC Classfication 2010:} 35L65, 37K10

\section{Introduction}
\label{Sec_Intro}
The Camassa-Holm equation, 

\begin{equation}
\label{Eq_CHoriginal}
u_t - u_{xxt} + 3uu_x = 2u_x u_{xx} + uu_{xxx},
\end{equation}
introduced in \cite{CH}, is an important model of fluid dynamics, which describes water waves in shallow canals. Here, $t \ge 0$ denotes time, $x \in \mathbb{R}$ is the space variable and $u(t,x): [0, \infty) \times \mathbb{R} \to \mathbb{R}$ is the 
horizontal velocity of water surface at time $t$ and position $x$ or, asymptotically equivalently, height of water surface over a flat bed, see \cite{CL} for a detailed discussion.
The key feature of the Camassa-Holm equation, which has brought about considerable interest of both physical and mathematical communities, is the fact that it accounts both for solitons (see \cite{CH}) and wave breaking phenomena (\cite{CE3}). In contrast, the celebrated KdV equation, used for modeling  similar phenomena, admits solitons, however preserves smoothness of solutions, thus prohibiting wave breaking (see e.g. \cite{KPV}). Solitons in the Camassa-Holm equation are peaked with an angle at their crest and for this reason are called peakons. Importantly, peaked waves are encountered also in the context of irrotational solutions of the Euler equations (whose approximation is \eqref{Eq_CHoriginal}), where they are called \emph{extreme Stokes waves} (see \cite{CIMA} and references therein). This and the fact that peakons as soliton solutions of \eqref{Eq_CHoriginal} are stable (and thus in principle observable, see \cite{CStrauss,Len0}) under small perturbations of their shape  provides further  rationale for the feasibility of the Camassa-Holm model as approximation of the Euler equations of inviscid fluid dynamics.

Well-posedness theory for Camassa-Holm was initiated in \cite{CE1}, where Constantin and Escher proved local well-posedness of \eqref{Eq_CHoriginal} in $H^3(\mathbb{R})$, showing however that, as hinted at in \cite{CH}, for a large class of initial data -- antisymmetric with $u_x(t=0,x=0)<0$ -- the solution blows up in finite time in such a way that $u_x(t,0) \to -\infty$ as $t \to T_{max}$. This type of blow-up has a clear physical interpretation of a wave profile, whose slope becomes steeper and steeper leading eventually to breaking in finite time.  On the other hand, in \cite{CE2} the same authors provided relatively general conditions for global existence of smooth solutions.

For smooth solutions, the Camassa-Holm equation can equivalently be expressed, by applying the operator $(I-\partial_{xx})^{-1}$ to \eqref{Eq_CHoriginal}, in the conservative form 
\begin{eqnarray}
\partial_t u + \partial_x (u^2 / 2) + P_x &=& 0, \label{eq_weakCH1}\\
P &=& \frac 1 2 e^{-|x|}* \left(u^2 + \frac {u_x^2}{2}\right),\\
u(t=0, \cdot) &=& u_0. \label{eq_weakCH3}
\end{eqnarray}

Formulation \eqref{eq_weakCH1}-\eqref{eq_weakCH3} is very convenient for studying weak solutions, whose introduction is indispensable if one wants to encompass the behaviour of a wave after the breaking time. It is also necessary to capture the non-smooth solitons (which are called peakons), see e.g. \cite{BC}. 
Indeed, one can see, that for general wave profiles with finite energy  ($u \in L^{\infty}([0,\infty),H^1(\mathbb{R}$))), the terms $u_x u_{xx}$ and $uu_{xxx}$ in \eqref{Eq_CHoriginal} are a priori ill-defined as distributions, whereas \eqref{eq_weakCH1}-\eqref{eq_weakCH3} makes perfect sense in this class of regularity. 

Using formulation \eqref{eq_weakCH1}-\eqref{eq_weakCH3} Xin and Zhang proved the following theorem.
\begin{theorem}(Existence of weak solutions, Xin-Zhang, \cite{XZ}). Assume that $u_0 \in H^1(\mathbb{R})$. Then the Cauchy problem \eqref{eq_weakCH1}-\eqref{eq_weakCH3} has an admissible weak solution $u = u(t,x)$ in the sense that
\begin{itemize}
\item $u(t,x) \in C([0,\infty) \times \mathbb{R}) \cap L^{\infty}([0,\infty), H^1(\mathbb{R})),$ 
\item $\|u(t,\cdot)\|_{H^1(\mathbb{R})} \le  \|u(0,\cdot)\|_{H^1(\mathbb{R})}$ for every $t > 0$,
\item $u(t,x)$ satisfies \eqref{eq_weakCH1}-\eqref{eq_weakCH3} in the sense of distributions.
\end{itemize}
\end{theorem}
\noindent Solutions constructed by Xin and Zhang by the method of vanishing viscosity satisfy
\begin{equation}
\label{eq_Oleinik}
\partial_x u(t,x) \le const \left(1+ \frac 1 t\right),
\end{equation}
which is reminiscent of the Oleinik entropy condition in conservation laws (see e.g. \cite[Section 3b]{EVANS}). Interestingly, direct use of \eqref{eq_Oleinik} in the proof of existence is not necessary, since (see \cite{CocK}) the vanishing viscosity method yields in fact strong compactness of the nonlinear terms. Admissible weak solutions of Camassa-Holm equation satisfying additionally \eqref{eq_Oleinik}  are called {\bf dissipative}. Uniqueness of dissipative solutions is an outstanding open problem. 

In \cite{BC2} a unique global semigroup of dissipative solutions was constructed, based on a special transformation of variables leading to a system of ordinary differential equations. As remarked by the authors, this does not provide uniqueness of solutions, since a different constructive procedure may lead to distinct dissipative solutions. Similar approach, based however on a transformation to Lagrangian variables, was followed by Holden and Raynaud in \cite{HR}.
On the other hand, uniqueness of the related class of conservative weak solutions, i.e. solutions conserving locally the energy (see e.g. \cite{BC}) was proven in \cite{BF}, based on construction of a Lipschitz metric for the semigroup of solutions. Analogous strategy  in Lagrangian variables was used by Grunert, Holden and Raynaud in \cite{HR2,GHR, GHR1} and a proof using generalized characteristics has been recently presented by Bressan, Chen and Zhang in \cite{BCZ}. Let us here also mention a recent paper \cite{GHR2}, where another special class of weak solutions, which interpolates between conservative and dissipative solutions was introduced and studied. Finally, let us note that one can study various limits of solutions of the Camassa-Holm equation, such as e.g. the convergence of solutions of \eqref{Eq_CHoriginal} to the entropy solutions of Burgers equation in the vanishing diffusion limit, see the recent paper by Coclite and di Ruvo \cite{CocR} and references therein.

Our approach is different from those in the above mentioned papers. Namely, we study \emph{arbitrary} admissible weak solutions of the Camassa-Holm equation without any further assumptions such as conservativeness or dissipativity. Denoting $f^+:= \max(f,0)$ and $f^- := \max(-f,0)$, we show that the 'square of derivative density' (SDD), $u_x^2$ (equal to the energy density, $u_x^2 + u^2$, up to the regular term $u^2$)
can be decomposed into the positive part, $(u_x^+)^2$ and negative part $(u_x^-)^2$, which satisfy:
\begin{itemize}
\item function $t \mapsto (u_x^+(t,\cdot))^2$ is weakly ladcag and of bounded variation, 
\item function $t \mapsto (u_x^-(t,\cdot))^2$ is weakly cadlag and of bounded variation.
\end{itemize} 
Here, 'ladcag' denotes functions which are left continuous with right limits and 'cadlag' -- right continuous with left limits.

This allows us to define the discrete (in time) part of the time-dependent dissipation/accretion measure for, separately, negative SDD and positive SDD by  
\begin{equation*}
\int_{\mathbb{R}} \phi(x) d\mu^+(t,dx) :=  \lim_{s \to t^+} \int_{\mathbb{R}} \phi(x) (u_x^+(s,x))^2 dx  -  \int_{\mathbb{R}} \phi(x) (u_x^+(t,x))^2 dx,
\end{equation*}
and 
\begin{equation*}
\int_{\mathbb{R}} \phi(x) d\mu^-(t,dx) :=  \int_{\mathbb{R}} \phi(x) (u_x^-(t,x))^2 dx - \lim_{s \to t^-} \int_{\mathbb{R}} \phi(x) (u_x^-(s,x))^2 dx,
\end{equation*}
where $\phi \in C_c(\mathbb{R})$ is an arbitrary continuous compactly supported test function. 
These two measures might be useful in obtaining better insights into the structure of solutions and thus approaching the uniqueness issues. The structure and properties of  measures $\mu^{\pm}$ as well as their relation to dissipativity are subject of a forthcoming paper, \cite{GJ}. Here let us only consider for illustration the example of conservative \emph{peakon-antipeakon} interaction. Namely, it turns out 
that (see \cite{BC} for a detailed discussion) the function of the form:
\begin{equation}
\label{Eq_upeakon}
u(t,x) = p_1(t) e^{-|x-q_1(t)|} - p_1(t)e^{-|x+q_1(t)|}
\end{equation}
where $q(0)<0, p_1(0)>0, p_1(t)= \frac 1 2 p(t), q_1(t) = \frac 1 2 q(t)$ and 
\begin{eqnarray*}
p(t) &=& H_0 \frac {[p(0) + H_0]+ [p(0) - H_0]e^{H_0 t}}{[p(0)+H_0] - [p(0)-H_0]e^{H_0 t}}, \\
q(t) &=& q(0) - 2 \log \frac {[p(0)+H_0]e^{-H_0t \slash 2} + [p(0) - H_0]e^{H_0t \slash 2}}{2p(0)},\\
H_0^2 &=& p(0)^2 (1 - e^{q(0)}) 
\end{eqnarray*}
is a weak solution of the Camassa-Holm equation. This solution has a singularity at time $T = \frac 1 {H_0} \log \frac{p(0)+H_0}{p(0)-H_0}$, which is characterized by 
\begin{eqnarray*}
\lim_{t \to T^-} \sup_x |u(t,x)| = 0
\end{eqnarray*}
and
\begin{eqnarray*}
\lim_{t \to T^-} e(t) = \lim_{t \to T^-} \frac 1 2 (u^2(t,x) + u_x^2(t,x)) = \lim_{t \to T^-} \frac 1 2  ((u_x^-)^2(t,x)) = H_0^2\delta_0(dx),
\end{eqnarray*}
where the limits are taken in the weak sense and $e(t)$ is the energy density. The solution can be prolonged beyond the blow-up time $T$ in a conservative way, i.e. so that it conserves locally the energy, by setting for $t>T$ 
$$u(t,x) = -u(2T- t,x).$$
We easily obtain then
\begin{eqnarray*}
\lim_{t \to T^+} e(t) = \lim_{t \to T^+} \frac 1 2 (u^2(t,x) + u_x^2(t,x))dx = \lim_{t \to T^+} \frac 1 2  ((u_x^+)^2(t,x))dx = H_0^2\delta_0(dx).
\end{eqnarray*}

Thus, the SDD of \eqref{Eq_upeakon} is a.e. weakly continuous (in fact it is continuous except for $t=0$, where it vanishes), which, however, does not imply that the discrete parts of dissipation/accretion measures vanish. Indeed, $\mu^+ = 2H_0^2\delta_0(dx)\delta_0(dt)$ and $\mu^- = -2H_0^2\delta_{0}(dx)\delta_0(dt)$, which means that at time $t=0$ the whole SDD is transferred from negative SDD to positive SDD as a singular package in $x=0$.
In this paper we show that \emph{any weak} solution of the Camassa-Holm equation, not necessarily the conservative one, can be studied in a similar fashion.

Another measure of dissipation can be, due to $BV$ regularity for fixed $\phi$, defined by
\begin{equation*}
\nu^+_{\phi}(dt) = \frac d {dt} \int_{\mathbb{R}} \phi(x)(u_x^+(t,x))^2 dx
\end{equation*} 
and, similarly, 
\begin{equation*}
\nu^-_{\phi}(dt) = \frac d {dt} \int_{\mathbb{R}} \phi(x)(u_x^-(t,x))^2 dx.
\end{equation*} 

Measures, $\nu_{\phi}^+$ and $\nu_{\phi}^-$ can, in contrast to measures $\mu^+, \mu^-$, account also for a continuous in time dissipation/accretion of energy, averaged out by function $\phi$. The examples of such continuous transfer of energy are stumpons and other travelling waves  considered in \cite{Len}. Let us again emphasize that, as we demonstrate in this paper, measures $\nu_{\phi}^+$, $\nu_{\phi}^-$ exist for \emph{any} weak solution of the Camassa-Holm equation.

To prove our results, we develop the theory of nonunique characteristics, initiated in the context of conservation laws by C. Dafermos, \cite{DafGC} and applied by him to the  Hunter-Saxton equation, \cite{DafCont, DafHS, DafMax}. This theory was then further developed by T. Cie\'{s}lak and the author in \cite{TCGJ} which led us to positive verification of the hypothesis of Zhang and Zheng, see \cite{ZhangZheng}, stating that dissipative solutions of the Hunter-Saxton equation dissipate energy at the highest possible rate. 
A similar question is also valid for the Camassa-Holm: 
\\ \, 

\emph{Does the principle of maximal energy/entropy dissipation (see \cite{DafMAX}) select the (unique) dissipative solution of the Camassa-Holm equation?}
\\ \, \\
 This question is widely open, not least due to lack of proof of uniqueness of dissipative solutions, but also because of technical complexity of studying maximal dissipation of energy/entropy, see  \cite{TCGJ} for the Hunter-Saxton case.

In the present paper, basing on the framework from \cite{TCGJ} we go beyond the Hunter-Saxton equation, and develop a more general framework for studying weak solutions by use of nonunique characteristics. In particular, considering equations of the form
\begin{equation*}
\mbox{(G) } u_t + uu_x = \int_{\mathbb{R}} A(x,y)[ au^2(t,y) + bu_x^2(t,y)]dy,
\end{equation*}
we include the Camassa-Holm equation, for which, in contrast to the Hunter-Saxton equation, $u_x$ may propagate nonlocally.
The hallmarks of our framework are:
\begin{itemize}
\item making use of transport of various quantities along highly non-unique characteristics,
\item proving certain properties for \emph{every} solution in a large class of weak solutions. 
\end{itemize}
The second aspect is particularly important for approaching questions involving the whole class of weak solutions such as the above mentioned criterion of maximal dissipation of entropy/energy.

To prove the weak cadlag/ladcag/BV properties of $(u_x^+)^2$ and $(u_x^-)^2$ and, consequently, existence of $\mu^{\pm}$ and $\nu^{\pm}$ we specialize, however, to the Camassa-Holm equation. This is due to the specific nature of the estimates. Nevertheless, similar considerations should be possible for related equations of the form (G) upon suitable modification of the arguments.

\, 

{\bf Acknowledgements.} The author was supported by the Swiss Government within a Swiss Government Excellence Scholarship for Foreign Scholars and Artists for the Academic Year 2015-2016. The author is furthermore grateful to his host, Gianluca Crippa, from the University of Basel, for supporting this research.

\section{Main results}
Let $A: \mathbb{R} \times \mathbb{R} \to \mathbb{R}$ be a bounded measurable function and $a,b \in \mathbb{R}$ satisfy $a \ge 0, b>0$. We consider the general equation
\begin{equation*}
\mbox{(G) } u_t + uu_x = \int_{\mathbb{R}} A(x,y)[ au^2(t,y) + bu_x^2(t,y)]dy
\end{equation*}
with the initial condition $u(t=0,x) = u_0(x)$. 
\begin{definition}[Weak solutions]
A function $u: [0,\infty)\times \mathbb{R} \to \mathbb{R}$ is called a weak solution of (G) if 
\begin{itemize}
\item $u(t,x) \in C([0,\infty) \times \mathbb{R}),$
\item $\int_{\mathbb{R}} (au^2(t,x) + bu_x^2(t,x))dx \in L^{\infty}([0,\infty)),$ 
\item $u(t=0,x) = u_0(x)$ for $x \in \mathbb{R}$,
\item $u(t,x)$ satisfies (G) in the sense of distributions.
\end{itemize}
\end{definition}
\noindent Let us  now impose the following assumptions on $A$. 
\begin{assumptions}
\label{Ass1}
\begin{eqnarray*}
\frac {A(x_2,y) - A(x_1,y)}{x_2-x_1} &\ge& -L, 
\end{eqnarray*}
and
for every $f \in L^1(\mathbb{R})$, the mapping $x \mapsto \int_{\mathbb{R}} A(x,y)f(y)dy$ is continuous.
\end{assumptions}

\begin{remark}
Coefficients $a,b,A$ for the Camassa-Holm equation and the related Hunter-Saxton equation satisfy $a\ge 0, b>0$ and Assumptions \ref{Ass1}, see Lemma \ref{Lem_Coef}.
\end{remark}

For weak solutions of (G) Assumptions \ref{Ass1} imply, by \cite[Lemma 3.1]{DafHS} the existence of characteristics. More precisely, we have the following straightforward result.
\begin{proposition}
\label{Prop_char}
Let $u$ be a weak solution of (G) under Assumptions \ref{Ass1}. Then for every $\zeta \in \mathbb{R}$ there exists a (nonunique) characteristic of $u$ emanating from $\zeta$, i.e. a function $\zeta : [0,\infty) \to \mathbb{R}$, which satisfies:
\begin{itemize}
\item $\zeta(0)=\zeta$,
\item $\frac d {dt} \zeta(t) = u(t,\zeta(t)),$
\item $\frac d {dt} u(t,\zeta(t)) = \int_{\mathbb{R}} A(\zeta(t),y)[ au^2(t,y) + bu_x^2(t,y)]dy.$
\end{itemize}
\end{proposition}
\noindent To introduce the technical results of this paper we need some definitions. For $\zeta, \eta \in \mathbb{R}$, $\eta \neq \zeta$ let $\zeta(t), \eta(t)$ be arbitrary characteristics emanating from $\zeta$ and $\eta$, respectively.
Define quantities
\begin{eqnarray*}
h(t) &:=& \eta(t) - \zeta(t),\\
p(t) &:=& u(t,\eta(t)) - u(t,\zeta(t)),\\
\omega(t) &:=& p(t)/h(t),\\
\end{eqnarray*}
and 
\begin{equation}
K_t(y) := \frac {A(\eta(t),y)-A(\zeta(t),y)}{h(t)}.
\label{Eq_Ktdef}
\end{equation}
Note that Assumptions \ref{Ass1} easily imply 
\begin{equation}
\label{Eq_Ktl}
K_t(y) \ge -L.
\end{equation}

In the following we will also consider sets $S_T \subset \mathbb{R}$ such that for every $\zeta \in S_T$ and $0\le t \le T$ the characteristic $\zeta(t)$ is unique and $K_t$, defined by \eqref{Eq_Ktdef}, can for $\eta > \zeta$ be decomposed as
\begin{equation}
K_t(y) = L(\zeta(t),y) + L_t^{1}(y)+L_t^{2}(y) + L_t^{3}(y),  \label{Eq_Kt}
\end{equation}
where $L$ is independent of $\eta(t)$ and for some real constants $C_1,C_2,C_3$
\begin{eqnarray}
L_t^{1}(y) &=& C_1\frac {1}{h(t)} \bold{1}_{[\zeta(t), \eta(t)]} (y), \label{Eq_L1} \\
|L_t^{2}(y)| &\le& C_2  \bold{1}_{[\zeta(t),\eta(t)]}(y),\label{Eq_L2}\\
|L_t^{3}(y)| &\le& C_3 h(t) \label{Eq_L3}. 
\end{eqnarray}
Our first result shows that not only $u$, but also $u_x$ evolves along characteristics.
\begin{proposition} 
\label{Prop_L}
Let $u$ be a weak solution of (G), where $(G)$ is such that Assumptions \ref{Ass1} are satisfied and that decomposition \eqref{Eq_Kt}-\eqref{Eq_L3} holds for any $y \in \mathbb{R}$,  any characteristics $\zeta(t),\eta(t)$ and universal constants $C_1,C_2,C_3$. 
Then there exists a family of sets $\{S_T\}_{T>0}$ such that for every $\zeta \in S_T$ and $0 \le t < T$
\begin{equation}
\label{Eq_PropL}
\dot{v}(t) = - v^2(t) + \int_{\mathbb{R}} L(\zeta(t),y)  (au^2(t,y) + bu_x^2(t,y))dy  +  C_1 (au^2(t,\zeta(t)) + bv^2(t)),
\end{equation}
where $v(t) := u_x(t,\zeta(t))$.
Moreover, $S_{T_1} \subset S_{T_2}$ for $T_1 > T_2$ and $|\mathbb{R} \backslash \bigcup_{T>0} S_T| = 0$.
\end{proposition}

Our second result depends more on the structure of the equation, so we formulate and prove it for the Camassa-Holm equation only. Nevertheless, the same strategy may be useful for studying related equations of the form (G).

\begin{theorem}
\label{Th_continuity}
Let $[\alpha,\beta] \subset \mathbb{R}$ be a compact interval. Let $u$ be a weak solution of the Camassa-Holm equation. 
Then
\begin{equation*}
\lim_{t \to 0^+} \int_{[\alpha(t),\beta(t)]} (u_x^-(t,x))^2 dx = \int_{[\alpha,\beta]} (u_x^-(0,x))^2 dx
\end{equation*}
where $\alpha(t)$ and $\beta(t)$ are arbitrary characteristics emanating from $\alpha$ and $\beta$, respectively.
\end{theorem}
Theorem \ref{Th_continuity} states that the part of $u_x^2$ related to negative slope of $u$ is right continuous along characteristics. The same result for positive slopes does not hold (see the example of the peakon-antipeakon interaction in Section \ref{Sec_Intro}). Nonetheless, it is possible to prove the existence of right limits. 

\begin{theorem}
\label{Th_limit}
Let $[\alpha,\beta] \subset \mathbb{R}$ be a compact interval. Let $u$ be a weak solution of the Camassa-Holm equation. 
Then for any pair of characteristics $\alpha(t)$ and $\beta(t)$ emanating from $\alpha$ and $\beta$ the function $t \mapsto \int_{[\alpha(t),\beta(t)]} (u_x^+(t,x))^2dx$ has locally bounded variation. In particular, there exists the limit
\begin{equation*}
\lim_{t \to 0^+} \int_{[\alpha(t),\beta(t)]} (u_x^+(t,x))^2 dx = \Lambda_+^{\alpha(\cdot), \beta(\cdot)}.
\end{equation*}
This limit depends on the choice of characteristics $\alpha(t),\beta(t)$ and can be made unique by selecting e.g. the leftmost characteristics (see Section \ref{Sec_Boundedness} for definition).
\end{theorem}

Theorems \ref{Th_continuity} and \ref{Th_limit} can be generalised by duality, see Section \ref{Sec_ProofThGeneral}, to left limits. We obtain the following general result. 
\begin{theorem}
\label{Th_general}
Let $u$ be a weak solution of the Camassa-Holm equation. Fix $t_0 \ge 0$. Let $\alpha(\cdot), \beta(\cdot)$ be any characteristics of $u$ satisfying $\alpha(t_0) < \beta(t_0)$. Then
\begin{eqnarray}
\lim_{t \to t_0^+} \int_{[\alpha(t),\beta(t)]} (u_x^-(t,x))^2 dx &=& \int_{[\alpha(t_0),\beta(t_0)]} (u_x^-(t_0,x))^2 dx, \label{EqTh81}\\
\lim_{t \to t_0^-} \int_{[\alpha(t),\beta(t)]} (u_x^+(t,x))^2 dx &=& \int_{[\alpha(t_0),\beta(t_0)]} (u_x^+(t_0,x))^2 dx. \label{EqTh82}
\end{eqnarray}
Moreover, functions $$t \mapsto \int_{[\alpha(t),\beta(t)]} (u_x^+(t,x))^2dx$$ and $$t \mapsto \int_{[\alpha(t),\beta(t)]} (u_x^-(t,x))^2dx$$ have locally bounded variation. In particular, there exist nonnegative numbers $\Lambda_+^{\alpha(\cdot), \beta(\cdot)}$, $\Lambda_-^{\alpha(\cdot), \beta(\cdot)}$ such that 
\begin{eqnarray} 
\lim_{t \to t_0^+} \int_{[\alpha(t),\beta(t)]} (u_x^+(t,x))^2 dx &=& \Lambda_+^{\alpha(\cdot), \beta(\cdot)}, \label{EqTh83}\\
\lim_{t \to t_0^-} \int_{[\alpha(t),\beta(t)]} (u_x^-(t,x))^2 dx &=& \Lambda_-^{\alpha(\cdot), \beta(\cdot)}. \label{EqTh84}
\end{eqnarray}
\end{theorem}
Using Theorem \ref{Th_general} we obtain, by suitable approximation with step functions, our main weak-continuity result.
\begin{theorem}
\label{Th_cadlag}
Let $u$ be a weak solution of the Camassa-Holm equation. Then
\begin{itemize}
\item function $t \mapsto (u_x^+(t,\cdot))^2$ is weakly ladcag (left-continuous with right limits) and $BV_{loc}$, 
\item function $t \mapsto (u_x^-(t,\cdot))^2$ is weakly cadlag (right-continuous with left limits) and $BV_{loc}$.
\end{itemize} 
More precisely, for any time $t_0 \ge 0$ and any continuous compactly supported function $\phi: \mathbb{R} \to [0,\infty)$ 
\begin{eqnarray*}
\lim_{t \to t_0^+} \int_{\mathbb{R}} \phi(x)(u_x^-(t,x))^2 dx &=& \int_{\mathbb{R}} \phi(x)(u_x^-(t_0,x))^2 dx,\\
\lim_{t \to t_0^-} \int_{\mathbb{R}} \phi(x) (u_x^+(t,x))^2 dx &=& \int_{\mathbb{R}} \phi(x)(u_x^+(t_0,x))^2 dx,\\
\end{eqnarray*}
and there exist limits
\begin{eqnarray*} 
\lim_{t \to t_0^+} \int_{\mathbb{R}} \phi(x)(u_x^+(t,x))^2 dx,\\
\lim_{t \to t_0^-} \int_{\mathbb{R}} \phi(x)(u_x^-(t,x))^2 dx.
\end{eqnarray*}
Moreover, the limits are linear in $\phi$ and so define bounded linear functionals on $C_c(\mathbb{R})$.
Finally, for any continuous compactly supported function $\phi: \mathbb{R} \to [0,\infty)$ the functions $$t \mapsto \int_{\mathbb{R}} \phi(x)(u_x^-(t,x))^2 dx$$ and $$t \mapsto \int_{\mathbb{R}} \phi(x)(u_x^+(t,x))^2 dx$$ have locally bounded variation on $[0,\infty)$.
\end{theorem}
Theorem \ref{Th_cadlag} can be then directly used to define measures $\mu^+$ and $\mu^-$ (by the Riesz representation theorem) as well as measures $\nu_{\phi}^{+}$ and $\nu_{\phi}^-$, as described in Section \ref{Sec_Intro} 

\section{Our framework and strategy of proof}
To prove Proposition \ref{Prop_L} we proceed in several steps.
\begin{enumerate}
\item For $\zeta, \eta \in \mathbb{R}$ let $\zeta(t), \eta(t)$ be arbitrary characteristics emanating from $\zeta$ and $\eta$, respectively. Recall the quantities 
\begin{eqnarray*}
h(t) &:=& \eta(t) - \zeta(t),\\
p(t) &:=& u(t,\eta(t)) - u(t,\zeta(t)),\\
\omega(t) &:=& p(t)/h(t).
\end{eqnarray*}
Then, as long as $h(t) \neq 0$, the quantities $h(t)$ and $p(t)$ satisfy, due to Proposition \ref{Prop_char},
\begin{eqnarray}
\dot{h} &=& p, \label{Eq_hpwh}\\
\dot{p} &=& \int_{\mathbb{R}} [A(\eta(t),y) - A(\zeta(t),y)][au^2(t,y) + bu_x^2(t,y)]dy,\label{Eq_hpwp} \\
\dot{\omega} &=& \frac {\dot{p}}{h} - \omega^2 = - \omega^2 + \frac {1}{h(t)} \int_{\mathbb{R}} [A(\eta(t),y)-A(\zeta(t),y)][au^2(t,y) + bu_x^2(t,y)] dy.\label{Eq_hpww}
\end{eqnarray}
Hence,
\begin{equation}
\omega(\tau) - \omega(\sigma) = - \int_{\sigma}^{\tau} \omega^2(t) + \int_{\sigma}^{\tau}\frac {1}{h(t)} \left( \int_{\mathbb{R}} [A(\eta(t),y)-A(\zeta(t),y)][au^2(t,y) + bu_x^2(t,y)] dy\right) dt,
\label{Eq_omega}
\end{equation}
which, by definition of $K_t$, \eqref{Eq_Ktdef}, can be expressed as
\begin{equation}
\omega(\tau) - \omega(\sigma) = - \int_{\sigma}^{\tau} \omega^2(t) + \int_{\sigma}^{\tau}\left( \int_{\mathbb{R}} K_t(y)[au^2(t,y) + bu_x^2(t,y)] dy\right) dt.
\label{Eq_omega2}
\end{equation}

\item In the following crucial step of the proof we pass to the limit $\eta \to \zeta$ in \eqref{Eq_omega}, to obtain that for a.e. $\zeta$ belonging to a certain set (called $L_T^{unique}$), 
\begin{eqnarray*}
v(\tau) - v(\sigma) = - \int_{\sigma}^{\tau} v^2(t) dt  + \int_{\sigma}^{\tau} \int_{\mathbb{R}} L(\zeta(t),y)  (au^2(t,y) + bu_x^2(t,y))dy dt \\ + \int_{\sigma}^{\tau} C_1 (au^2(t,\zeta(t)) + bu_x^2(t,\zeta(t)))dt,
\end{eqnarray*}
where $0 \le \sigma \le \tau < T$ and $v(\rho) := u_x (\rho, \zeta(\rho))$. In this passage, the critical role is played by various properties of characteristics, which we demonstrate. This step  (especially the main technical result, Lemma \ref{Lem5}) follows the lines of \cite{TCGJ} adapted to a new setting.
\item We conclude that for a.e. $\zeta \in L_T^{unique}$ the derivative $u_x$ evolves, for $t \in [0,T)$, along the characteristics according to the equation:
\begin{equation}
\dot{v}(t) = - v^2(t) + \int_{\mathbb{R}} L(\zeta(t),y)  (au^2(t,y) + bu_x^2(t,y))dy  +  C_1 (au^2(t,\zeta(t)) + bv^2(t)),
\label{vzkropka}
\end{equation}
which proves Proposition \ref{Prop_L}.
\end{enumerate}
Next, we use equation \eqref{vzkropka} to prove Theorem \ref{Th_continuity} as follows.

\begin{enumerate}
\item Equation \eqref{vzkropka} for the Camassa-Holm equation reads: 
$$\dot{v} = u^2 - \frac 1 2 v^2 - P.$$
\item Using this equation we estimate the quantity
$$\left| \int_{[\alpha(t),\beta(t)]} (u_x^-(t,x))^2 dx   - \int_{[\alpha,\beta]} (u_x^-(0,x))^2 dx \right|$$ 
by decomposing intervals $[\alpha,\beta]$ and $\alpha(t),\beta(t)$ with respect to different properties of respective characteristics (unique, unique with bounded difference quotients, other). The central role is again played by properties of characteristics, which we further investigate. Since the estimates are not invariant with respect to the sign of $v$, the fact that the \emph{negative} part of derivative is considered is vital.
\end{enumerate}
To prove Theorem \ref{Th_limit} we show that for $t_2>t_1>0$
\begin{equation*}
\int_{[\alpha(t_2),\beta(t_2)]} (u_x^+(t_2,x))^2 dx \ge \int_{[\alpha(t_1),\beta(t_1)]} (u_x^+(t_1,x))^2 dx + f(t_2 - t_1) 
\end{equation*}
for some sufficiently regular function $f$. This means that the function $$t \mapsto \int_{[\alpha(t),\beta(t)]} (u_x^+(t,x))^2 dx$$ is increasing up to a regular correction which precludes wild oscillations and guarantees existence of the left limit. Since $f$ can be chosen Lipschitz continuous, we obtain in fact $BV_{loc}$ regularity of $t \mapsto \int_{[\alpha(t),\beta(t)]} (u_x^+(t,x))^2 dx$.

Theorem \ref{Th_general} is a relatively straightforward consequence of Theorems \ref{Th_continuity} and \ref{Th_limit} and the fact that for a weak solution $u$ of the Camassa-Holm equation functions $-u(t_0 - t, x)$ and 
$u(t-t_0,x)$ are also weak solutions.

Finally, Theorem \ref{Th_cadlag} follows from Theorem \ref{Th_general} by suitable approximation of function $\phi$ by step functions.

The remaining part of the paper is organised as follows. 
In Section \ref{Sec_fitin} we show that the Camassa-Holm equation and Hunter-Saxton equation fit into our framework. In Section \ref{Sec_Boundedness} we show the crucial properties and estimates on characteristics. In particular, we prove that almost every (in suitable sense) characteristic has the uniqueness property. In Section \ref{Sec_PropL} we obtain the equation for evolution of $u_x$ along characteristics, thus proving Proposition \ref{Prop_L}. In Section \ref{Sec_ContPrel} we prove some auxiliary results, which are applied in Sections \ref{Sec_ProofCont}, \ref{Sec_Th9}, \ref{Sec_ProofThGeneral} and \ref{Sec_Th10} to prove Theorems \ref{Th_continuity}, \ref{Th_limit}, \ref{Th_general} and \ref{Th_cadlag}, respectively.

\section{Camassa-Holm and Hunter-Saxton fit in}
\label{Sec_fitin}
The Camassa-Holm equation, \eqref{eq_weakCH1}, can be written in the form
\begin{equation*}
\mbox{(C-H) }  u_t + uu_x = \int_{\mathbb{R}} \frac 1 2 sgn(x-y)e^{-|x-y|} (u^2(t,y) + \frac 1 2 u_x^2(t,y)) dy 
\end{equation*}
whereas the Hunter-Saxton equation as 
\begin{equation*}
\mbox{(H-S) } u_t + uu_x = \int_{\mathbb{R}} \bold{1}_{(-\infty,x]} (y) \frac 1 2 {u_x^2}(t,y) dy.
\end{equation*}
Hence, they both can be regarded as special cases of equation 
\begin{equation*}
\mbox{(G) } u_t + uu_x = \int_{\mathbb{R}} A(x,y)[ au^2(t,y) + bu_x^2(t,y)]dy
\end{equation*}
with 
\begin{eqnarray*}
\mbox{(C-H) }&& A(x,y) = \frac 1 2 sgn(x-y)e^{-|x-y|}, \quad  a=1, \quad b=\frac 1 2\\
\mbox{(H-S) }&& A(x,y) = \bold{1}_{(-\infty,x]}(y) = \bold{1}_{[0,\infty)}(x-y), \quad a=0, \quad b=\frac 1 2
\end{eqnarray*}
One can note that both for H-S and C-H the kernel $A(x,y)$ can be expressed in the form $A(x,y) = \tilde{A}(x-y)$ for certain $\tilde{A}$.

\begin{lemma}
\label{Lem_Coef}
For both H-S and C-H we have $a\ge 0, b>0$ and coefficients $A$ satisfy Assumptions \ref{Ass1}. 
\end{lemma}
\begin{proof}
The conditions $a \ge 0, b>0$ are obvious.  
To show that $A$  satisfies Assumptions \ref{Ass1} we proceed as follows. Let $x_1 < x_2$. Then for Hunter-Saxton
\begin{equation*}
A(x_2,y) - A(x_1,y) = \bold{1}_{[x_1,x_2]}(y) \ge 0.
\end{equation*}
For Camassa-Holm, on the other hand, 
\begin{eqnarray*}
A(x_2,y) - A(x_1,y) &=& \frac 1 2 [ sgn(x_2 -y) e^{-|x_2-y|} - sgn(x_1 -y) e^{-|x_1-y|}]  \\
&=& \frac 1 2  (sgn(x_2 - y) - sgn(x_1 - y)) e^{-|x_2 - y|} +  \frac 1 2 sgn(x_1 - y) (e^{-|x_2-y|} - e^{-|x_1-y|}).
\end{eqnarray*}
Now, the first term is nonnegative, and the second is bounded from below by $-(x_2-x_1)$ due to the fact that $e^{-|x|}$ is Lipschitz with constant $1$. Finally, for $f \in L^1(\mathbb{R})$ the continuity of $$x \mapsto \int_{-\infty}^x f(y)dy$$ is obvious, whereas continuity of $$x \mapsto \int_{\mathbb{R}} \frac 1 2 sgn(x-y)e^{-|x-y|} f(y)dy$$ follows by the Lebesgue dominated convergence theorem.
This proves continuity of the mapping $x \mapsto \int_{\mathbb{R}} A(x,y)f(y)dy$ for $f \in L^1(\mathbb{R})$ for C-H and H-S.
\end{proof}
Consequently, by Proposition \ref{Prop_char}, weak solutions of both H-S and C-H  satisfy the system of characteristics
\begin{equation*}
\begin{cases}
\dot{x} = u, \\
\dot{u} = \int_{\mathbb{R}} A(x,y) [au^2(y) + bu_x^2(y)]dy.
\end{cases}
\end{equation*}

\begin{lemma}
$K_t$ for the Camassa-Holm and Hunter-Saxton equations can be, for $\eta>\zeta$, expressed in the form \eqref{Eq_Kt}-\eqref{Eq_L3}.
\end{lemma}
\begin{proof}
The case of Hunter-Saxton is straightforward. For Camassa-Holm we obtain
\begin{eqnarray*}
K_t(x) &=& \frac {A(\eta(t),y)-A(\zeta(t),y)}{h(t)}\\
&=& \frac 1 {2h(t)} \left( sgn(\eta(t)-y)e^{-|\eta(t)-y|} -  sgn(\zeta(t)-y) e^{-|\zeta(t)-y|} \right) \\
&=& \frac {1}{2h(t)} \left( (sgn(\eta(t)-y) - sgn(\zeta(t)-y)) e^{-|\eta(t)-y|} + sgn(\zeta(t)-y) (e^{-|\eta(t)-y|} - e^{-|\zeta(t)-y|})  \right)\\ 
&=& - \frac 1 2 e^{-|\zeta(t)-y|} + K_t^{11}(y)+K_t^{12}(y) - K_t^{21}(y) - K_t^{22}(y),
\end{eqnarray*}
where
\begin{eqnarray*}
K_t^{11}(y) &=& \frac {1}{h(t)} \bold{1}_{[\zeta(t), \eta(t)]} (y), \\ 
K_t^{12}(y) &=& \frac {1}{h(t)} \bold{1}_{[\zeta(t), \eta(t)]} (y) (e^{- |\eta(t) -y |}-1),\\ 
K_t^{21}(y) &=& \frac{1}{2h(t)} \int_{\zeta(t)}^{\eta(t)} (sgn(\zeta(t)-y)sgn(s-y) )(e^{-|s-y|}  -  e^{-|\zeta(t)-y|})ds, \\
K_t^{22}(y) &=&  \frac{1}{2h(t)} \int_{\zeta(t)}^{\eta(t)} (sgn(\zeta(t)-y)sgn(s-y) -1) e^{-|\zeta(t)-y|}ds.
\end{eqnarray*}
Indeed,
\begin{equation*}
K_t^{11}(y) + K_t^{12}(y) = \frac {1}{h(t)} \bold{1}_{[\zeta(t), \eta(t)]} e^{- |\eta(t) -y |} = \frac {1}{2h(t)} (sgn(\eta(t)-y) - sgn(\zeta(t)-y)) e^{-|\eta(t)-y|}
\end{equation*}
and
\begin{eqnarray*}
&&- \frac 1 2 e^{-|\zeta(t)-y|} - K_t^{21}(y) - K_t^{22}(y)\\
&=& - \frac 1 2 e^{-|\zeta(t)-y|} -\frac{1}{2h(t)} \int_{\zeta(t)}^{\eta(t)} (sgn(\zeta(t)-y)sgn(s-y) e^{-|s-y|}  -  e^{-|\zeta(t)-y|})ds \\
&=& -\frac{1}{2h(t)} \int_{\zeta(t)}^{\eta(t)} sgn(\zeta(t)-y)sgn(s-y) e^{-|s-y|}ds \\
&=& \frac{sgn(\zeta(t)-y)}{2h(t)} \int_{\zeta(t)}^{\eta(t)} -sgn(s-y) e^{-|s-y|}ds \\
&=& \frac{sgn(\zeta(t)-y)}{2h(t)} \int_{\zeta(t)}^{\eta(t)} \frac d {ds} e^{-|s-y|}ds \\
&=&\frac{sgn(\zeta(t)-y)}{2h(t)} (e^{-|\eta(t)-y|} - e^{-|\zeta(t)-y|}).
\end{eqnarray*}

For these terms we easily obtain the following estimates
\begin{eqnarray*}
|K_t^{12}| &\le& \bold{1}_{[\zeta(t),\eta(t)]}(y),\\ 
|K_t^{21}|&\le& \frac 1 2 \sup_{s \in [\zeta(t),\eta(t)]} |e^{-|s-y|} - e^{-|\zeta(t) - y|}| \le h(t), \\
|K_t^{22}| &\le&   \bold{1}_{[\zeta(t),\eta(t)]}(y).\\
\end{eqnarray*}
Hence, denoting
\begin{eqnarray*}
L(\zeta(t),y) &:=& - \frac 1 2 e^{- |\zeta(t)-y|},\\
L_t^1(y) &:=& K_t^{11}(y),\\
L_t^2(y) &:=& K_t^{12}(y) + K_t^{22}(y),\\
L_t^3(y) &:=& K_t^{21}(y)
\end{eqnarray*}
we conclude.
\end{proof}

\section{Boundedness of $\omega$ along characteristics}
\label{Sec_Boundedness}
In this section we prove the most important results related to characteristics. Let us begin by introducing the leftmost and rightmost characteristics.

\begin{lemma}
\label{Lem_CharRL}
Let $u: [0,T] \times  \mathbb {R} \to \mathbb{R}$ be a locally bounded continuous function. Let $\{x_{\alpha}\}_{\alpha \in A}$ be a family of functions satisfying, for $t \in [0,T]$,
\begin{eqnarray*}
\dot{x}_{\alpha}(t) &=& u(t,x_{\alpha}(t)),\\
x_{\alpha}(0) &=& x.
\end{eqnarray*}
Then function $y(t):= \sup_{\alpha \in A} x_{\alpha}(t)$ satisfies
\begin{equation*}
\dot{{y}}(t) = u(t,y(t)).
\end{equation*}
Similarly, function $z(t):= \inf_{\alpha \in A} x_{\alpha}(t)$ satisfies
\begin{equation*}
\dot{{z}}(t) = u(t,z(t)).
\end{equation*}
\end{lemma}
\begin{proof}
Function $y(t)$ is Lipschitz continuous as supremum of a family of uniformly Lipschitz continuous functions (see e.g. \cite[Proposition 3.5]{TCGJ}).  
Fix $t \in (0,T)$ and $\epsilon > 0$. Let $\delta>0$ be so small that for every $(s,x) \in [t-\delta,t+\delta] \times [y(t)- 2\delta u(t,y(t)),y(t)+ 2\delta u(t,y(t))]$ we have $|u(s,x) - u(t,y(t))|<\epsilon$.
Then for any $\alpha$ such that $x_{\alpha}(t)>y(t)- \frac \delta {10}  u(t,y(t))$ and  $s \in (t-\delta,t+\delta)$ we obtain 
$$|x_{\alpha}(s) - x_{\alpha}(t) - (s-t)u(t,y(t))| \le |s-t| \sup_r {|u(r,x_\alpha(r)) - u(t,y(t))|} < \epsilon|s-t|.$$
Consequently,
$$y(s) \ge x_{\alpha}(s) \ge x_{\alpha}(t) + (s-t)u(t,y(t)) - \epsilon |s-t|.$$
Taking the supremum over $\alpha$ we obtain
$$y(s) \ge y(t) + (s-t)u(t,y(t)) - \epsilon |s-t|.$$
Hence, if $y(\cdot)$ is differentiable in $t$ then
$$u(t,y(t)) - \epsilon \le \dot{y}(t) \le u(t,y(t)) + \epsilon.$$
By arbitrariness of $\epsilon$ we conclude. The proof for $z(t)$ is analogous.
\end{proof}

\begin{corollary}
\label{Cor_lefrightchar}
\rm
Let $u$ be a weak solution of (G). Then for every $t_0\ge 0$ and $\zeta \in \mathbb{R}$ there exists the rightmost characteristic emanating from $\zeta$ at time $t_0$. It is defined as the unique characteristic $\zeta^r(t)$, which satisfies $\zeta^r(t_0)=\zeta$ and $\zeta^r(t) \ge \zeta(t)$ for every $t\ge t_0$ and every characteristic $\zeta(\cdot)$ emanating from $\zeta$ at time $t_0$. Similarly, there exists the leftmost characteristic, $\zeta^l$ emanating from $\zeta$ at time $t_0$, which is defined analogously.  
\end{corollary}
\begin{proof}
Immediate consequence of Lemma \ref{Lem_CharRL} combined with \cite[Lemma 3.1]{DafHS}.
\end{proof}
Next, let us make precise the notion of unique forwards and unique backwards characteristic.
\begin{definition}
\label{Def_UniqueChar}
Let $u$ be a weak solution of (G). 
\begin{itemize}
\item A characteristic $\zeta(\cdot)$ of $u$ on $[0,T]$ is called \emph{unique forwards} if for every characteristic $\eta(\cdot)$ of $u$ such that $\eta(t_0) = \zeta(t_0)$ for some $t_0 \in [0,T]$ we have $\eta(t) = \zeta(t)$ for every $t \in [t_0,T]$, see Fig. \ref{Fig2} left.
\item  A characteristic $\zeta(\cdot)$ of $u$ on $[0,T]$ is called \emph{unique backwards} if for every characteristic $\eta(\cdot)$ of $u$ such that $\eta(t_0) = \zeta(t_0)$ for some $t_0 \in [0,T]$ we have $\eta(t) = \zeta(t)$ for every $t \in [0,t_0]$, see Fig. \ref{Fig2} right. 
\end{itemize}
\end{definition}

\begin{figure}[h!]
\center
\includegraphics[width=8cm]{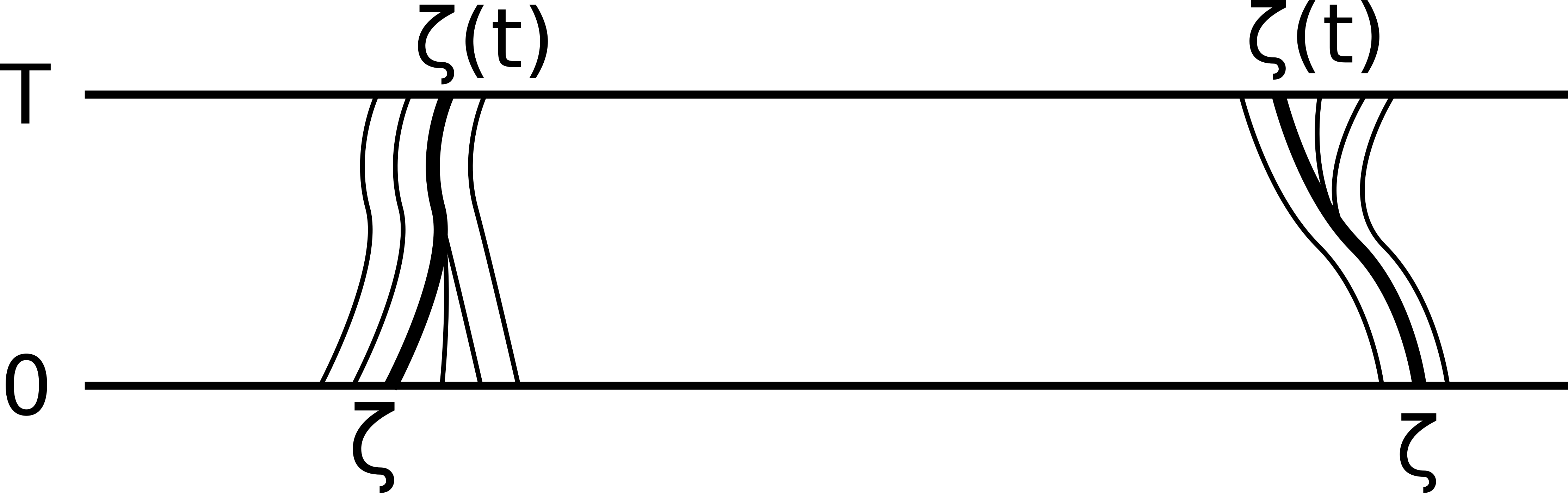}
\caption{Schematic presentation of unique forwards characteristic (left) and unique backwards characteristic (right). Unique forwards characteristics have no branching sites yet may collide with other characteristics. Contrarily, unique backwards do not collide with other characteristics yet may branch.}
\label{Fig2}
\end{figure}

\noindent Let now $\zeta, \eta \in \mathbb{R}$. 
Recalling formulas \eqref{Eq_hpwh}-\eqref{Eq_omega2} we obtain 
\begin{eqnarray}
\dot{\omega} &=& \frac {\dot{p}}{h} - \omega^2 = - \omega^2 + \frac {1}{h(t)} \int_{\mathbb{R}} [A(\eta(t),y)-A(\zeta(t),y)][au^2(t,y) + bu_x^2(t,y)] dy \nonumber \\
&=& - \omega^2 + \int_{\mathbb{R}} K_t(y)[au^2(t,y) + bu_x^2(t,y)] dy. \label{Eq_omegadot}
\end{eqnarray}
Our goal is to identify the sets $S_T \ni \zeta$ on which it is possible to pass to the limit $\eta \to \zeta$ in \eqref{Eq_omegadot}. 
To this end, we define $$C:= \sup_{t \in [0,\infty)} \int_{\mathbb{R}} [au^2(t,y)+bu_x^2(t,y)]dy$$
and begin with an elementary estimate of $\omega$.
\begin{lemma}
\label{Lem2}
Under Assumptions \ref{Ass1}
\begin{eqnarray}
\dot{\omega} &\ge& -\omega^2 - LC, \label{Eq_1}\\
\omega(t_1) &\ge& \sqrt{LC} \tan \left( -\sqrt{LC} (t_1-t_0) + \arctan\left(\frac {\omega(t_0)}{\sqrt{LC}}\right)\right) \label{Eq_2}
\end{eqnarray}
for $0 \le t_0 < t_1$.
\end{lemma}

\begin{proof}
The first inequality follows easily from \eqref{Eq_Ktl} and \eqref{Eq_omegadot} while for the second we calculate:
\begin{eqnarray*}
\dot{\tilde{\omega}} &=& -{\tilde \omega}^2 - LC, \\
\frac{d\tilde{\omega}}{\tilde{\omega}^2 + LC} &=& -dt,\\
\frac {1}{LC} \frac {d\tilde{\omega}}{1+ (\frac {\tilde{\omega}}{\sqrt{LC}})^2} &=& - dt,\\
\frac {1}{\sqrt{LC}} \left[\arctan\left(\frac{\tilde{\omega}(t_1)}{\sqrt{LC}}\right) - \arctan\left(\frac{\tilde{\omega(t_0)}}{\sqrt{LC}}\right)\right] &=& -(t_1-t_0),\\
\tilde{\omega}(t_1) &=& {\sqrt{LC}}\tan \left[ - {\sqrt{LC}}(t_1-t_0) +  \arctan\left(\frac{\tilde{\omega}(t_0)}{\sqrt{LC}}\right)\right].
\end{eqnarray*}
Moreover,
\begin{eqnarray*}
\frac d {dt} (\tilde{\omega} - \omega) \le \omega^2 - \tilde{\omega}^2 = -(\omega+\tilde{\omega})(\tilde{\omega} - \omega) \le K |\tilde{\omega} - \omega|.
\end{eqnarray*}
Taking the initial condition $\tilde{\omega}(t_0) -\omega(t_0) = 0$ and using the Gronwall inequality we obtain $\tilde{\omega} - \omega \le 0$, which proves \eqref{Eq_2}.
\end{proof}

\begin{definition}
\label{Def_OI}
Let  $T_{max} := \frac {\pi}{8\sqrt{LC}}$ and for every $t \in [0,T_{max})$ define 
\begin{eqnarray*}
\Omega(t) &:=& \sqrt{LC} \tan (\sqrt{LC}t - \frac \pi 2), \\
I_t &:=& \{\zeta \in \mathbb{R}: u_x(0,\zeta) \mbox{ is a limit of difference quotients and } u_x(0,\zeta) > \Omega(t) \},\\
I_t^{unique,N} &:=& \{\zeta \in I_t:  \forall_{\eta \in (\zeta- \frac 1 N,\zeta) \cup (\zeta,\zeta+ \frac 1 N),  s \in [0,t)} -N \le \omega(s) \le N \},\\
I_t^{unique} &:=& \bigcup_{N=1}^{\infty} I_t^{unique,N}.
\end{eqnarray*}
\end{definition}

The idea behind these definitions is the following. $I_t$ is (see Proposition \ref{Prop_It}) the set of $\zeta$, which give rise to characteristics which are certainly unique backwards until time $t$ (however, it does not encompass all such unique characteristics, compare Definition \ref{Def_Lt}). The set $I_t^{unique} \subset I_t$ contains characteristics which are additionally unique forwards and have $\omega$ bounded, which could correspond to a considerably smaller set. The main technical result regarding these sets (Lemma \ref{Lem5}) is that $I_t$ and $I_t^{unique}$ are in fact equal up to a set of measure $0$.

\begin{proposition} 
\label{Prop_It}
Sets $I_t$ and $I_t^{unique,N}$ have the following properties.
\begin{enumerate}[a)]
\item If $\zeta \in I_t$ then there exists $N\in \mathbb{N}$ such that $\omega(s) > -N$ for every $\eta \in (\zeta- \frac 1 N, \zeta) \cup (\zeta,\zeta+ \frac 1 N)$ and  $s \in [0,t]$. 
\item If $\zeta \in I_t$ then $\zeta(t)$ is unique backwards on $[0,t]$.
\item For $\zeta \in I_t^{unique,N}$, $\eta \in (\zeta - \frac 1 N, \zeta) \cup (\zeta,\zeta+ \frac 1 N)$ and $s \in [0,t]$ we have $$|\eta-\zeta|e^{-Ns} \le |\eta(s) - \zeta(s)| \le |\eta-\zeta|e^{Ns}.$$
\item For $\zeta \in I_t^{unique,N}$ the characteristic $\zeta(t)$ is unique forwards and backwards.\\
\end{enumerate}
\end{proposition}

\begin{proof}
If $\zeta \in I_t$ then $u_x(0,\zeta)>\Omega(t)$ and, consequently, $\omega(0)>\Omega(t)$ for $|\eta - \zeta|$ small enough. Using formula \eqref{Eq_2} we obtain a). To prove b) we observe that $\dot{h} = \omega h$. Condition $\omega(s)>-N$ leads then to $h(s) \slash h(0) \ge e^{-Ns}$. The proof of c) is analogous, while d) follows from c). Indeed, if a characteristic starting from $\zeta$ is nonunique then there exist characteristics $\zeta_1(\cdot), \zeta_2(\cdot)$ satisfying $\zeta_1(0) = \zeta_2(0) = \zeta$ and $\zeta_1(s_0) < \zeta_2(s_0)$ for some $0 < s_0 < t$. Hence, $0< \zeta_2(s_0) - \zeta_1(s_0) \le \eta(s_0) - \zeta_1(s_0) \le (\eta - \zeta) e^{Ns_0}$, where $\eta(\cdot)$ is the rightmost characteristic emanating from $\eta>\zeta$. Taking $\eta \to \zeta^+$, we obtain a contradiction.
\end{proof}

\begin{lemma}
\label{Lem5}
Let $u$ be a weak solution of (G). Fix $\tau<T_{max}$. Then $|I_{\tau} \backslash \bigcup_{N \in \mathbb{N}} I_{\tau}^{unique,N}| = 0$. Equivalently, $I_{\tau} = I_{\tau}^{unique} = \bigcup_{N \in \mathbb{N}} I_{\tau}^{unique,N}$, where both equalities hold up to a set of measure $0$.

\end{lemma}
\begin{proof}
Observe that $I_{\tau} = \bigcup_{\{t>\tau, t \in \mathbb{Q}\}} I_{t}$. Hence, it is enough to show that for fixed $t \in (\tau,T_{max})$
$$|I_{t} \backslash \bigcup_{N \in \mathbb{N}} I_{\tau}^{unique,N}| = 0.$$
To this end, define

\begin{eqnarray*}
J_t^{bad} &:=& \{\zeta \in I_t: \forall_{\epsilon>0} \forall_{M>0}\exists_{\eta\in (\zeta,\zeta+\epsilon)}\exists_{s \in [0,\tau]}  \omega^{\zeta,\eta}(s)>M  \},\\
\tilde{J}_t^{bad} &:=& \{\zeta \in I_t: \forall_{\epsilon>0} \forall_{M>0}\exists_{\eta\in (\zeta-\epsilon, \zeta) }\exists_{s \in [0,\tau]}  \omega^{\zeta,\eta}(s)>M  \},
\end{eqnarray*}
where $\omega^{\zeta,\eta}$ is calculated for any characteristic $\eta(s)$ starting from $\eta$.
To conclude, it suffices, in view of Proposition \ref{Prop_It}a, to show that $|J_t^{bad}| = 0$ and $|\tilde{J}_t^{bad}| = 0$. 

\noindent For this purpose, define, for every $\zeta,M,\delta$,
$$\Pi_{\zeta}^{M,\delta} := [\zeta,\eta],$$
where $\eta$ is such that $\eta \in (\zeta,\zeta+\delta)$ and there exists $s \in [0,\tau]$ such that $\omega(s) > M$. 
Then 
$$\mathcal{E}^M := \{ \Pi_{\zeta}^{M,\delta}, \delta>0, \zeta \in J_t^{bad}\}$$
is a covering of $J_t^{bad}$. By the Vitali covering theorem we obtain an at most countable  pairwise disjoint family of closed intervals $\mathcal{F}^M \subset \mathcal{E}^M$ such that 
$$J_t^{bad} \subset \bigcup \mathcal{F}^M$$
holds up to a set of measure $0$.
$\mathcal{F}^M$ can be represented as $\bigcup_{i=1}^{\infty}\{[\zeta_i,\eta_i]\}$ with fixed timepoints $s_i$ satisfying $\omega^{\zeta_i,\eta_i}(s_i) > M$, for some fixed characteristics $\eta_i(s)$. 
Using the formula $\tan(\alpha+\beta) = \frac {\tan(\alpha) + \tan(\beta)}{1-\tan(\alpha)\tan(\beta)}$ we obtain, by \eqref{Eq_2}, 
\begin{eqnarray*}
\omega(t_1) &\ge& \sqrt{LC} \frac {\tan ( - \sqrt{LC} (t_1-t_0)) + \frac {\omega(t_0)}{\sqrt{LC}}}{1 - \tan(-\sqrt{LC}(t_1-t_0)) \frac {\omega(t_0)}{\sqrt{LC}}} \\ 
&=& \sqrt{LC} \frac {\sin( - \sqrt{LC} (t_1-t_0)) + \frac {\omega(t_0)}{\sqrt{LC}} \cos( - \sqrt{LC} (t_1-t_0))}{ \cos( - \sqrt{LC} (t_1-t_0)) - \frac {\omega(t_0)}{\sqrt{LC}}\sin( - \sqrt{LC} (t_1-t_0))}.
\end{eqnarray*}
Above, we used also the fact that for fixed $\lambda<t$ and $t_0<t_1\le \lambda <t$ we have 
\begin{equation}
\label{Eq_tan}
\tan(-\sqrt{LC}(t_1-t_0)) \frac {\omega(t_0)}{\sqrt{LC}} < 1-\kappa(t,\lambda)
\end{equation}
for some $\kappa(t,\lambda)>0$ dependent on $\lambda,t$, see definition of $I_t$. Hence, for $t_0 < t_2 \le t$,
\begin{eqnarray*}
\int_{t_0}^{t_2} \omega(t_1)dt_1 \ge  \ln \left( \cos( - \sqrt{LC} (t_2-t_0)) - \frac {\omega(t_0)}{\sqrt{LC}} \sin( - \sqrt{LC} (t_2-t_0)) \right),
\end{eqnarray*}
which means that 
\begin{equation}
\label{Est_h}
h(t_2) = h(t_0)e^{\int_{t_0}^{t_2} \omega(t_1)dt_1} \ge h(t_0) \left[ \cos( - \sqrt{LC} (t_2-t_0)) - \frac {\omega(t_0)}{\sqrt{LC}} \sin( - \sqrt{LC} (t_2-t_0))\right].
\end{equation}
Taking $t_0:= s_i, t_2:=t$ and denoting $\omega_i(s):= \omega^{\zeta_i,\eta_i}(s)$, $h_i(s) := \eta_i(s) - \zeta_i(s)$ we arrive at
\begin{eqnarray*}
h_i(t) &\ge& h_i(s_i) \left[\cos( - \sqrt{LC} (t-s_i)) - \frac {\omega_i(s_i)}{\sqrt{LC}} \sin( - \sqrt{LC} (t-s_i))\right] \\
&\ge& h_i(s_i) \frac {\omega_i(s_i)}{\sqrt{LC}} \sin(\sqrt{LC} (t-s_i)) \ge \frac M 2 h_i(s_i) (t-\tau),
\end{eqnarray*}
where we used the estimate $\sin(x) \ge \frac 1 2 x$ for $x \le \frac \pi 2$. 
Let us now assume that  
$M>\sqrt{LC}\tan\left(\frac 3 8 \pi\right)$. Then, by Lemma \ref{Lem2}, $\omega_i(t) \ge \sqrt{LC}$ and by \eqref{Eq_tan} and \eqref{Est_h} with $t_2:=s_i$ and $t_0:=0$, we have $h_i(s_i)\ge h_i(0) \kappa \cos(\pi/4)$. Using these inequalities, we  calculate
\begin{eqnarray*}
  \int_{\mathbb{R}} u_x^2(t,x)dx &\ge& \sum_{i=1}^{\infty} \int_{\zeta_i(t)}^{\eta_i(t)} u_x^2(t,x)dx\\
&\ge& \sum_{i=1}^{\infty} (\eta_i(t)-\zeta_i(t)) \frac {1}{\eta_i(t)-\zeta_i(t)} \int_{\zeta_i(t)}^{\eta_i(t)} u_x^2(t,x)dx\\
&\ge& \sum_{i=1}^{\infty} (\eta_i(t)-\zeta_i(t)) \left[\frac {1}{\eta_i(t)-\zeta_i(t)} \int_{\zeta_i(t)}^{\eta_i(t)} u_x(t,x)dx\right]^2\\
&\ge& \sum_{i=1}^{\infty} h_i(t) \left[\frac {p_i(t)}{h_i(t)}\right]^2 \\
&=& \sum_{i=1}^{\infty} h_i(t) \omega_i^2 (t) \\
&\ge& LC\sum_{i=1}^{\infty} h_i(t) \\
&\ge& \frac {LCM} 2 \sum_{i=1}^{\infty} h_i(s_i) (t-\tau) \\
&\ge& \frac {\kappa(t,\tau)LCM\cos(\pi/4)(t-\tau)} 2 \sum_{i=1}^{\infty} |\eta_i - \zeta_i| \\
&\ge & \frac {\kappa(t,\tau) LCM\cos(\pi/4)(t-\tau)} 2 |J_t^{bad}|.
\end{eqnarray*}
Since $\int_{\mathbb{R}} u_x^2(t,x)dx$ is bounded by the definition of weak solutions, taking $M \to \infty$ we conclude that $|J_t^{bad}| = 0$. The proof for $\tilde{J}_t^{bad}$ is analogous.
\end{proof}
As mentioned before, set $I_t$ contains \emph{not all} starting points of characteristics unique until time $t$. Similarly, set $I_t^{unique,N}$ does not contain all starting points of characteristics, for which $\omega$ is bounded by $N$. To encompass all (up to a set of measure $0$, see Lemma \ref{lem_IL}) such characteristics, we introduce a second family of sets. 

\begin{definition}
\label{Def_Lt}
\begin{eqnarray*}
L_t^{unique,N}&:=& \{\zeta \in \mathbb{R}: \zeta(s) \mbox{ is unique forwards on $[0,t]$, } \\ && \zeta(s) \mbox{ is a Lebesgue point of } u_x(s,\cdot) \mbox{ for almost every } s \in [0,t] \mbox{ and }\\ &&\forall_{\eta \in (\zeta- \frac 1 N, \zeta) \cup (\zeta,\zeta+ \frac 1 N),  s \in [0,t)} -N \le \omega(s) \le N \} ,\\
L_t^{unique} &:=& \bigcup_{N=1}^{\infty}L_t^{unique,N}.
\end{eqnarray*}
\end{definition}

\begin{lemma}
\label{lem_IL}
For every $N \in \mathbb{N}$ we have $|I_t^{unique,N} \backslash L_t^{unique,N}|=0$ and hence $|I_t^{unique} \backslash L_t^{unique}|=0.$
\end{lemma}
\begin{proof}
Using Fubini theorem one can show that for almost every $\zeta \in I_t^{unique,N}$ point $\zeta(s)$ is a Lebesgue point of $u_x(s,\cdot)$ for almost every $s$ (see \cite[Page 166]{DafHS}).
\end{proof}

Our goal in the remaining part of the section is the following convergence result, whose proof follows the lines of \cite{TCGJ}.
\begin{theorem}
\label{Th_conv}
Let $u$ be a weak solution of (G) and let $f \in L^{\infty}([0,T], L^1(\mathbb{R}))$. Then there exists a sequence $\epsilon_k \to 0$ such that for almost every $\zeta \in L_T^{unique}$ and $0 \le \sigma \le \tau < T$ we have
\begin{equation}
\label{Eq_toprove0th}
\limsup_{k \to \infty}  \left| \int_{\sigma}^{\tau} \left( \frac {1}{(\zeta+\epsilon_k)^l(t) - \zeta(t)} \int_{\zeta(t)}^{(\zeta+\epsilon_k)^l(t)} f(t,y)dy - f(t,\zeta(t)) \right) dt \right| = 0,
\end{equation}
\end{theorem}
where $(\zeta + \epsilon_k)^l(t)$ is the leftmost characteristic emanating from $\zeta + \epsilon_k$, see Corollary \ref{Cor_lefrightchar}.
\begin{remark}
Since, by Lemma \ref{lem_IL}, $I_T^{unique} \subset L_T^{unique}$ (up to a set of measure $0$), the convergence \eqref{Eq_toprove0th} holds for almost every $\zeta \in I_T^{unique}$.  
\end{remark}

To prove Theorem \ref{Th_conv} we introduce a more convenient notation.

\begin{definition}
\label{Def_Mt}
Let $u$ be a weak solution of (G). For $\zeta \in \mathbb{R}$ define function $M_t(\zeta)$ by
$$M_t(\zeta):=\zeta^l(t),$$
\end{definition}
where $\zeta^l(t)$ is the leftmost characteristic of $u$ emanating from $\zeta$, see Corollary \ref{Cor_lefrightchar}.

\begin{proposition}
\label{Prop14}
For every $\zeta \in L_t^{unique,N}$ and $\epsilon < 1 \slash N$ we have
\begin{enumerate}[i)]
\item $M_t(\zeta+\epsilon) - M_t(\zeta) \le \epsilon e^{tN}$,
\item $M_t(\zeta+\epsilon) - M_t(\zeta) \ge \epsilon e^{-tN}$,
\item $M_t'(\zeta) \ge e^{-tN}.$
\end{enumerate}
\end{proposition}
\begin{proof}
Using equation $h(t) = h(0)e^{\int_0^t \omega(s)ds}$ we obtain $\epsilon e^{-tN} \le M_t(\zeta+\epsilon) - M_t(\zeta) \le \epsilon e^{tN}$, which proves i) and ii). Dividing ii) by $\epsilon$ and taking $\liminf$ we obtain iii).
\end{proof}

\begin{proposition}
\label{Prop_General54}
Let $u$ be a weak solution of (G). Fix $N\in \mathbb{N}$. Let $c(\tau), C(\tau)$ be functions (possibly dependent on $N$) satisfying $0<c(\tau) \le C(\tau) < \infty$ for every $\tau \in [0,T)$ such that:\begin{enumerate}[i)]
\item $M_t'(\zeta) \ge c(\tau),$
\item $M_t(\zeta+\epsilon) - M_t(\zeta) \le \epsilon C(\tau),$
\item $M_t(\zeta+\epsilon) - M_t(\zeta) \ge \epsilon c(\tau)$
\end{enumerate}
for every $\zeta \in L_T^{unique,N}$ and $0 \le t \le \tau < T$. Let $f \in L^{\infty}([0,T], L^1(\mathbb{R}))$. Then for almost every $\zeta \in L_T^{unique,N}$ and $0 \le \sigma \le \tau < T$ we have
\begin{equation}
\label{Eq_toprove0}
\limsup_{k \to \infty}  \left| \int_{\sigma}^{\tau} \left( \frac {1}{(\zeta+\epsilon_k)^l(t) - \zeta(t)} \int_{\zeta(t)}^{(\zeta+\epsilon_k)^l(t)} f(t,y)dy - f(t,\zeta(t)) \right) dt \right| = 0.
\end{equation}
for some sequence $\epsilon_k \to 0$.
\end{proposition}
\begin{proof}
We show that 
\begin{equation}
\label{Eq_toprove}
\lim_{\epsilon \to 0} \int_{L_T^{unique,N}}  \int_{\sigma}^{\tau} \left( \frac {1}{(\zeta+\epsilon)^l(t) - \zeta(t)} \int_0^{(\zeta+\epsilon)^l(t) - \zeta(t)} |f(\zeta(t)+y,t) - f(\zeta(t),t)|dy \right)dt d\zeta = 0
\end{equation}
Indeed,
\begin{eqnarray*}
&&\int_{L_T^{unique,N}} \left( \frac {1}{(\zeta+\epsilon)^l(t) - \zeta(t)} \int_0^{(\zeta+\epsilon)^l(t) - \zeta(t)} |f(\zeta(t)+y,t) - f(\zeta(t),t)|dy \right) d\zeta \\
&=&\int_{L_T^{unique,N}} \left( \frac {1}{M_t(\zeta+\epsilon) - M_t(\zeta)} \int_0^{M_t(\zeta+\epsilon) - M_t(\zeta)} |f(\zeta(t)+y,t) - f(\zeta(t),t)|dy \right) d\zeta \\ 
&\le&\int_{L_T^{unique,N}} \left( \frac {1}{c(\tau)\epsilon} \int_0^{C(\tau)\epsilon} |f(\zeta(t)+y,t) - f(\zeta(t),t)|dy \right) d\zeta \\ 
&=& \int_{L_T^{unique,N}} g^{\epsilon}(M_t(\zeta))d\zeta =: S_N^{\epsilon}(t),
\end{eqnarray*}
where $g^{\epsilon}(z) := \frac {1}{c(\tau)\epsilon} \int_0^{C(\tau)\epsilon} |f(z+y,t)-f(z,t)|dy$. Now, since $g^{\epsilon}$ is bounded, nonnegative and Borel measurable, we obtain, using the change of variables in the Stjeltjes integral (see \cite{FT}) and neglecting the singular part of measure $dM_t$,
\begin{eqnarray*}
\int_{M_t\left(L_T^{unique,N}\right)} g^{\epsilon}(z)dz &=& \int_{L_T^{unique,N}} g^{\epsilon}(M_t(\zeta))dM_t(\zeta) \\ &\ge& \int_{L_T^{unique,N}} g^{\epsilon}(M_t(\zeta)) M_t'(\zeta) d\zeta \ge c(\tau) \int_{L_T^{unique,N}} g^{\epsilon}(M_t(\zeta))d\zeta.
\end{eqnarray*}
Hence,
\begin{eqnarray*}
S_N^{\epsilon}(t) &\le& \frac {1}{c(\tau)} \int_{M_t\left(L_T^{unique,N}\right)} g^{\epsilon}(z)dz \\
&=& \frac {1}{c(\tau)} \int_{M_t\left(L_T^{unique,N}\right)} \frac {1}{c(\tau)\epsilon} \int_0^{C(\tau)\epsilon} |f(z+y,t)-f(z,t)|dydz \\
&\le& \frac {C(\tau)}{c^2(\tau)} \int_{\mathbb{R}} \frac {1}{C(\tau)\epsilon} \int_0^{C(\tau)\epsilon} |f(z+y,t)-f(z,t)|dydz. \\
\end{eqnarray*}
Now, since $f(\cdot,t) \in L^1(\mathbb{R})$ we obtain 
\begin{equation*}
\lim_{\epsilon \to 0} S_N^{\epsilon}(t) = 0.
\end{equation*}
Moreover, $S_N^{\epsilon}(t) \le 2 \frac {C(\tau)}{c^2(\tau)} \|f\|_{L^{\infty}([0,T],L^1(\mathbb{R}))}$. This implies, by the Lebesgue dominated convergence theorem,
\begin{equation*}
\lim_{\epsilon \to 0^+} \int_{\sigma}^{\tau} S^{\epsilon}_N(t)dt = 0
\end{equation*}
which, after extraction of a subsequence, proves \eqref{Eq_toprove}. 
\end{proof}

\begin{proof}[Proof of Theorem \ref{Th_conv}]
The proof follows immediately by use of Propositions \ref{Prop14}, \ref{Prop_General54} and a diagonal argument (compare \cite[Proposition 5.3]{TCGJ}).
\end{proof}

\section{Proof of Proposition \ref{Prop_L}}
\label{Sec_PropL}
Let $S_T := L_T^{unique} \backslash Z_T$, where $Z_T$ is such that $|Z_T|=0$ and Theorem \ref{Th_conv} holds for every $\zeta \in S_T$. Let $\zeta \in S_T$. Integrating \eqref{Eq_omegadot} from $\sigma$ to $\tau$ we obtain
\begin{eqnarray*}
\omega(\tau) - \omega(\sigma) &=& - \int_{\sigma}^{\tau} \omega^2(t) + \int_{\sigma}^{\tau}\frac {1}{h(t)} \int_{\mathbb{R}} [A(\eta(t),y)-A(\zeta(t),y)][au^2(t,y) + bu_x^2(t,y)] dy dt\\
&=& - \int_{\sigma}^{\tau} \omega^2(t) + \int_{\sigma}^{\tau} \int_{\mathbb{R}} K_t(y)[au^2(t,y) + bu_x^2(t,y)] dy dt\\
&=& - \int_{\sigma}^{\tau} \omega^2(t) dt\\
&&+ \int_{\sigma}^{\tau} \int_{\mathbb{R}} L(\zeta(t),y)[au^2(t,y) + bu_x^2(t,y)] dy dt\\ 
&&+ \int_{\sigma}^{\tau} \int_{\mathbb{R}} C_1 \frac{1}{h(t)}  \bold{1}_{[\zeta(t), \eta(t)]} (y)[au^2(t,y) + bu_x^2(t,y)] dy dt\\
&& + \int_{\sigma}^{\tau} \int_{\mathbb{R}} L_t^2(y)[au^2(t,y) + bu_x^2(t,y)] dy dt \\
&& + \int_{\sigma}^{\tau} \int_{\mathbb{R}} L_t^3(y)[au^2(t,y) + bu_x^2(t,y)] dy dt\\
&=& T_1 + \int_{\sigma}^{\tau} \int_{\mathbb{R}} L(\zeta(t),y)[au^2(t,y) + bu_x^2(t,y)] dy dt +  T_3 + T_4 + T_5.
\end{eqnarray*}
Since  $\zeta \in S_T$ we have the following convergences as $k \to \infty$, where $\eta = \zeta + \epsilon_k$ and $\epsilon_k$ was constructed in Theorem \ref{Th_conv}:
\begin{itemize}
\item $\eta^l(s) \to \zeta(s)$ uniformly on $[0,T]$ by Proposition \ref{Prop14},
\item $\omega(\rho) \to u_x(\rho,\zeta(\rho))$ for almost every $\rho \in [0,T)$ due to the fact that $\zeta(t)$ is a Lebesgue point of $u_x(t,\cdot)$ for almost every $t$,
\item $T_1 \to -\int_{\sigma}^{\tau} u_x^2(t,\zeta(t)) dt$ by boundedness of $\omega$ and the Lebesgue dominated convergence theorem,
\item $T_3 \to C_1\int_{\sigma}^{\tau} (au^2(\zeta(t),t) + bu_x^2(\zeta(t),t))dt $ by Theorem \ref{Th_conv} with $f=au^2+b u_x^2$,
\item $T_4 \to 0$ by estimate $$|T_4| \le  \int_{\sigma}^{\tau} \int_{\mathbb{R}} C_2 \bold{1}_{[\zeta(t),\eta(t)]}(y)[au^2(t,y) + bu_x^2(t,y)] dy dt\le  \frac {C_2} {C_1} \sup_{t \in [0,T]} ( (\zeta+\epsilon_k)(t) - \zeta(t)) |T_3|,$$
\item $T_5 \to 0$ by estimate $$|T_5| \le C_3 \sup_{t \in [0,T]} h(t)  \int_{\sigma}^{\tau} \int_{\mathbb{R}} [au^2(t,y) + bu_x^2(t,y)] dydt,$$
\end{itemize}
where we used estimates \eqref{Eq_L1}-\eqref{Eq_L3}. Using these convergences, we obtain 
\begin{eqnarray*}
u_x(\tau,\zeta(\tau)) - u_x(\sigma,\zeta(\sigma)) = -\int_{\sigma}^{\tau} u_x^2(t,\zeta(t))dt &+& \int_{\sigma}^{\tau} \int_{\mathbb{R}} L(\zeta(t),y)[au^2(t,y) + bu_x^2(t,y)] dydt \\
&+&  C_1\int_{\sigma}^{\tau} (au^2(t,\zeta(t)) + bu_x^2(t,\zeta(t)))dt,
\end{eqnarray*}
which is equivalent to \eqref{Eq_PropL}. Inclusion $S_{T_1} \subset S_{T_2}$ for $T_1 > T_2$ is obvious by construction and $|\mathbb{R} \backslash \bigcup_{T>0} S_T| = 0$ follows by Lemma \ref{Lem5}, Lemma \ref{lem_IL} and the fact that $\bigcup_{t>0} I_t = \mathbb{R}$.

\begin{corollary}
\label{CorCHODE}
A weak solution, $u$, of Camassa-Holm satisfies
\begin{equation}
\label{Eq_dotv}
\dot{v} = u^2 - \frac 1 2 v^2 - P
\end{equation}
where $v(t)=u_x(\zeta(t),t)$, for every $t \in [0,T]$ provided that $\zeta \in S_T$. Moreover, if  
$\zeta \in L_T^{unique,N}$ then $|v| \le N$ on $[0,T]$ and thus for every $\zeta \in S_T$ there exists $N$ such that $|v| \le N$ on $[0,T]$.
\end{corollary}

\section{Continuity theorems for Camassa-Holm -- preliminary results}
\label{Sec_ContPrel}

\noindent{\bf Change of variables formula}

If $g$ is a bounded nonnegative Borel measureable function then (see \cite{FT}) for $0 \le t \le T$
\begin{eqnarray*}
\int_{M_t\left(A\right)} g(z)dz &=& \int_{A} g(M_t(\zeta))dM_t(\zeta).
\end{eqnarray*}
If now $A \subset L_t^{unique}$ then for every $\zeta \in A$ there exists $N$ such that $$\forall_{\eta \in (\zeta- \frac 1 N, \zeta) \cup (\zeta,\zeta+ \frac 1 N),  s \in [0,t)} -N \le \omega(s) \le N.$$
Hence, for $\eta \in (\zeta- \frac 1 N, \zeta) \cup (\zeta,\zeta+ \frac 1 N)$ we have $$e^{-Nt} \le \frac {M_t(\eta) - M_t(\zeta)}{\eta - \zeta} \le e^{Nt}.$$
This implies that  the singular part of measure $dM_t$  satisfies $dM_t^{singular} (A) = 0$ and hence $dM_t$ can be replaced by $M_t' (\zeta) d\zeta$. We obtain
\begin{eqnarray*}
\int_{M_t\left(A \right)} g(z)dz &=& \int_{A} g(M_t(\zeta))dM_t(\zeta) \\ &=& \int_{A } g(M_t(\zeta)) M_t'(\zeta) d\zeta.
\end{eqnarray*}

\noindent {\bf Estimate of ${\bf v^2M_t'}$}
\begin{proposition}
\label{Prop_doublebound}
Let $\zeta(\cdot)$, with $\zeta \in L_T^{unique}$ be a unique characteristic of (C-H)  satisfying for some positive constants $V_1, V_2$,
\begin{itemize}
\item $\dot{v} = u^2 - \frac 1 2 v^2 - P$
\item $-V_2 \le v \le -V_1$ or $V_1 \le v \le V_2$ on $[0,T]$,
\end{itemize}
where $v(t) = u_x(t,\zeta(t))$ and $u(t) = u(t,\zeta(t))$.
Then for every $t \in [0,T]$ we have 
\begin{equation}
\label{Eq_BigN}
e^{-\frac {2t}{V_1}((\sup u)^2 + \sup(P))} \le \frac {v^2(t) M_{t}'(\zeta)}{v^2(0) M_{0}'(\zeta)} \le e^{\frac {2t}{V_1}((\sup u)^2 + \sup(P))} ,
\end{equation}
where $M_t(\eta) = \eta^l(t)$ and $\eta^l(t)$ is the unique leftmost characteristic emanating from $\eta$.
\end{proposition}
\begin{remark}
Estimate \eqref{Eq_BigN} is independent of $V_2$.
\end{remark}

\begin{proof}
We have $\dot{h} = \omega h$  and hence $h(t) = h(0)e^{\int_0^t \omega(s)ds}$.
Using the fact that $\eta(s) \to \zeta(s)$ for every $s \in [0,t]$ and boundedness of $\omega$ on $[0,t]$ we obtain
\begin{equation}
\label{Eq_Mprim}
M'_t (\zeta) = \lim_{\eta \to \zeta} \frac {h(t)}{h(0)} = \lim_{\eta \to \zeta} e^{\int_0^t \omega(s)ds} = e^{\int_0^t v(s)ds},
\end{equation}
where we used Lebesgue dominated convergence theorem to pass to the limit. Thus, we have the equations
\begin{eqnarray*}
\dot{v} &=& u^2 - \frac 1 2 v^2 - P,\\
M'_t(\zeta)&=& e^{\int_0^t v(s)ds},
\end{eqnarray*}
which can be transformed into 
\begin{eqnarray*}
\dot{v} &=& -\frac 1 2 v^2 \left[1+\frac {u^2 - P}{-1/2 v^2}\right],\\
\frac{d}{dt}{M_t'(\zeta)} &=& v(t)M_t'(\zeta).
\end{eqnarray*}
Hence (below, for the sake of readability we skip the indexes $t,\zeta$ by $M$),
\begin{eqnarray*}
\frac d {dt}(v^2 M') = 2v\dot{v}M' + v^2 \dot{M'} = -v^3\left[1+\frac {u^2 - P}{-1/2 v^2}\right] M' + v^3 M' = 2v(u^2 - P)M'.
\end{eqnarray*}
Since both $v$ and $M'$ are bounded and bounded away from $0$, we can calculate the logarithmic derivative:
\begin{equation*}
\frac d {dt} \ln(v^2M') = \frac {1}{v^2M'} \frac d {dt} (v^2 M') = \frac {2(u^2 - P)}{v},
\end{equation*}
which means that
\begin{equation*}
(v^2 M')(t) = (v^2 M')(0) e^{\int_0^t \frac {2(u^2-P)}{v}}.
\end{equation*}
Estimate \eqref{Eq_BigN} follows.
\end{proof}

\noindent{\bf Symmetry of equation}
\begin{lemma}
\label{Lem_bwd}
Let $u$ be a weak solution of (G):
\begin{equation*}
u_t + uu_x = \int_{\mathbb{R}} A(x,y) [au^2(y) + bu_x^2(y)]dy.
\end{equation*}
Then the function $u^{bT}(t,x):= -u(T-t,x)$ satisfies equation
\begin{equation*}
u^{bT}_t + u^{bT}u^{bT}_x = \int_{\mathbb{R}} A(x,y) [a{(u^{bT})}^2(y) + b{(u^{bT}_x})^2(y)]dy.
\end{equation*}
with initial condition $u^{bT}(0,x) = -u(T,x)$.
\end{lemma}
\begin{proof}
Substitution.
\end{proof}

\section{Proof of Theorem \ref{Th_continuity}}
\label{Sec_ProofCont}
\begin{definition}
Let $u$ be a weak solution of (C-H). We define
$$D_t^{unique,N} := \{ \zeta \in L_t^{unique}: \forall_{s \in [0,t]} |u_x(s,\zeta(s))| \le N \}$$
\end{definition}
\noindent By Corollary \ref{CorCHODE} we obtain $L_t^{unique,N} \subset D_t^{unique,N}$. 

\begin{remark}
In general it is not clear whether $L_t^{unique,N} = D_t^{unique,N}$. Strict inclusion might be due to the following effect. Consider, for instance, the function 
\begin{equation*}
u(t,x) = \begin{cases} 0 &\mbox{ if } x/t < 1/2\\
t(x/t - 1/2) &\mbox{ if } 1/2 \le x/t \le 3/2\\
t(3/2 - x/t) &\mbox{ if } 1 \le x/t \le 3/2\\
0 &\mbox{ if } x/t > 3/2
\end{cases}
\end{equation*}
Then although $u_x(t,0) = 0$ for every $t>0$, the best possible uniform bound for difference quotients on the half-line $(t,0), t>0$ is $1/2$.
\end{remark}

\noindent Let now $u$ be a weak solution of (C-H) and fix $\epsilon >0$. Moreover, let $\alpha(\cdot), \beta(\cdot)$ be the leftmost characteristics emanating from $\alpha,\beta$, respectively (for general case see the end of this section). 

{\bf Choose ${\bf T>0}$} so small that 
\begin{itemize}
\item 
\begin{equation}
\label{Eq_es1}
\int_{[\alpha,\beta]\backslash I_T} u_x^2(0,\zeta)d\zeta < \epsilon,
\end{equation}
\item
\begin{equation}
\label{Eq_es111}
T \le  \frac {\pi}{8 \sqrt{LC}} 
\end{equation}
\item
\begin{equation}
\label{Eq_es11}
T \le \epsilon \frac {\sqrt{LC}}{2} (( \sup (u^2) + \sup(P))^{-1} 
\end{equation}
\item $$\Omega(T) < - 4 ((\sup u)^2 + \sup(P))^2,$$
\end{itemize}
where $\Omega(T)$ and $I_T$ are defined in Definition \ref{Def_OI} and the suprema are taken over the compact set confined by the characteristics $\alpha(\cdot), \beta(\cdot)$ and lines $t=0$ and $t=T$   (see Fig. \ref{Fig1}).
\begin{figure}[h]
\center
\includegraphics[width=5.5cm]{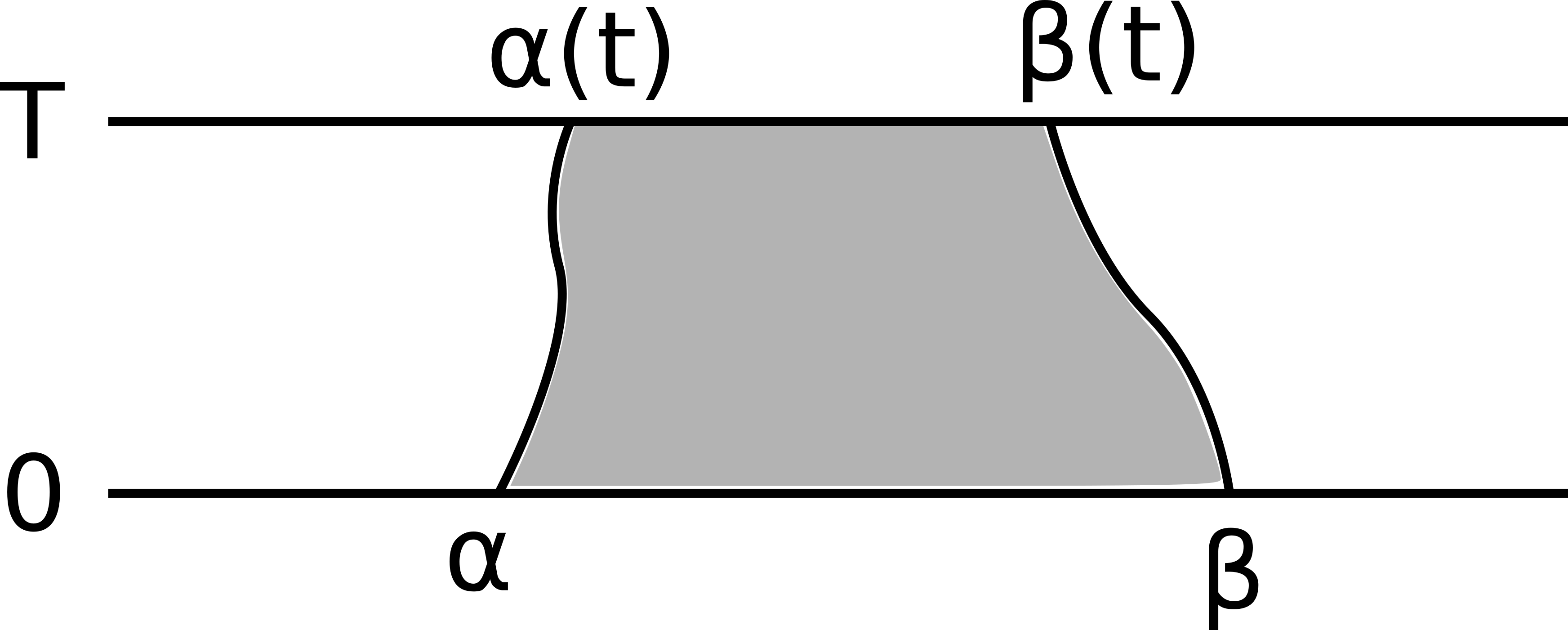}
\caption{Schematic drawing of the area considered in the proof of Theorem \ref{Th_continuity}. The area is confined by characteristics $\alpha(\cdot), \beta(\cdot)$ and straight lines $\{t=0\}, \{t=T\}$.}
\label{Fig1}
\end{figure}

Note that under these assumptions we have $$\int_{[a,b]\backslash I_t} u_x^2(0,\zeta)d\zeta < \epsilon$$ for $t \in [0,T]$ since $I_T \subset I_t$ for $t \in [0,T]$. 

\noindent{\bf Next, choose ${\bf N}$} so that:
\begin{itemize}
\item $$N > \max(1,(2((\sup u)^2 + \sup(P))^2),$$ 
\item  
\begin{equation}
\label{Eq_es12}
N \ge T (\sup(u)^2 + \sup(P)) + \sqrt{LC}
\end{equation}

\item $N$ is so large that (using Lemma \ref{Lem5})
\begin{equation}
\label{Eq_es2}
\int_{I_T \backslash I_T^{unique,N}} u_x^2(0,\zeta)d\zeta < \epsilon
\end{equation}
and, consequently, $\int_{I_T \backslash L_t^{unique,N}} u_x^2(0,\zeta)d\zeta < \epsilon$  and $\int_{I_T \backslash D_t^{unique,N}} u_x^2(0,\zeta)d\zeta < \epsilon$ for every $t \in [0,T]$ (due to inclusions $I_T^{unique,N} \subset L_T^{unique,N} \subset L_t^{unique,N} \subset D_t^{unique,N}$),

\item 
\begin{equation}
\label{Eq_NLC}
N \ge \sqrt{LC} \max( \tan (3\pi / 8) , \tan(\sqrt{LC}T)).
\end{equation}

\end{itemize}
{\bf Finally, choose ${\bf t_\epsilon<T}$} so small that 
\begin{itemize}
\item 
\begin{equation}
\label{Eq_eNt}
|e^{Nt_\epsilon} - 1| < \epsilon, \qquad
|e^{-Nt_\epsilon} - 1| < \epsilon,
\end{equation}
\item 
\begin{equation}
\label{Eq_v2}
t_\epsilon < \frac \epsilon {2N (\sup(u)^2 + \sup(P)+ \frac 1 2 N^2)}.
\end{equation}
\end{itemize} 
Now, {\bf for fixed ${\bf t \in [0,t_\epsilon]}$}, decompose 
\begin{eqnarray*}
[\alpha,\beta] &=& (D_t^{unique,N} \cap [\alpha,\beta])\cup ((I_T  \backslash D_t^{unique,N} )\cap [\alpha,\beta]) \cup ([\alpha,\beta] \backslash I_T), \\ 
\ [\alpha(t),\beta(t)]&=&M_t(D_t^{unique,N} \cap [\alpha,\beta]) \cup M_t((L_t^{unique} \backslash D_t^{unique,N}) \cap [\alpha,\beta]) \cup ([\alpha(t),\beta(t)] \backslash M_t(L_t^{unique}))
\end{eqnarray*}
 and estimate:
\begin{eqnarray*}
&&\left| \int_{[\alpha(t),\beta(t)]} (u_x^-(t,z))^2 dz   - \int_{[\alpha,\beta]} (u_x^-(0,\zeta))^2 d\zeta \right| \\
&\le& \left| \int_{M_t(D_t^{unique,N} \cap [\alpha,\beta])} (u_x^- (t,z))^2 dz - \int_{D_t^{unique,N}\cap [\alpha,\beta]} (u_x^-(0,\zeta))^2 d\zeta \right|\\
&& + \int_{[\alpha,\beta] \backslash I_T} (u_x^-(0,\zeta))^2 d\zeta\\
&& + \int_{(I_T  \backslash D_t^{unique,N} )\cap [\alpha,\beta]} (u_x^-(0,\zeta))^2 d\zeta\\ 
&& + \int_{M_t((L_t^{unique} \backslash D_t^{unique,N}) \cap [\alpha,\beta])} (u_x^- (t,z))^2 dz \\
&& + \int_{[\alpha(t),\beta(t)] \backslash M_t(L_t^{unique})} (u_x^- (t,z))^2 dz  = {\bf I_1} + {\bf I_2} + {\bf I_3} + {\bf I_4} + {\bf I_5}.
\end{eqnarray*}

\noindent{\bf Term ${\bf I_1}$}

\noindent First, observe that for $\zeta \in D_t^{unique,N}$ we have 
\begin{equation}
\label{Eq_absvdot}
|\dot{v}| \le \sup(u)^2 + \sup(P)+ \frac 1 2 N^2.
\end{equation}
Hence, by $|v| \le N$ and \eqref{Eq_absvdot} $$|v^2(t) - v^2(s)| \le |v(t)-v(s)||v(t)+v(s)| \le 2tN (\sup(u)^2 + \sup(P)+ \frac 1 2 N^2),$$
which for $0 \le s<t<t_{\epsilon}$ gives, by \eqref{Eq_v2} 
\begin{equation}
\label{Eq_v2epsilon}
|v^2(t) - v^2(s)| \le \epsilon
\end{equation}
and
\begin{equation}
\label{Eq_vepsilon}
|v(t) - v(s)| \le \frac {\epsilon}{2N}.
\end{equation}
To estimate $I_1$ we calculate:
\begin{eqnarray*} 
I_1&=&\left|\int_{M_t(D_t^{unique,N}\cap [\alpha,\beta])} (u_x^-(t,z))^2 dz  - \int_{D_t^{unique,N}\cap [\alpha,\beta]} (u_x^-(0,\zeta))^2d\zeta \right| \\ 
&\le& \left|\int_{M_t(D_t^{unique,N}\cap [\alpha,\beta]\cap \{u_x(0,\cdot)<0\})} u_x^2(t,z)dz  - \int_{D_t^{unique,N}\cap [\alpha,\beta]\cap \{u_x(0,\cdot)<0\})} u_x^2(0,\zeta)d\zeta \right|\\
&&+\int_{M_t(D_t^{unique,N}\cap [\alpha,\beta]\cap \{u_x(0,\cdot)<0\})} (u_x^+(t,z))^2dz + \int_{M_t(D_t^{unique,N}\cap [\alpha,\beta]\cap \{u_x(0,\cdot)\ge 0\})} (u_x^-(t,z))^2dz \\
&\le& \left|\int_{M_t(D_t^{unique,N}\cap [\alpha,\beta]\cap \{u_x(0,\cdot)<0\})} u_x^2(t,z)dz  - \int_{D_t^{unique,N}\cap [\alpha,\beta]\cap \{u_x(0,\cdot)<0\})} u_x^2(0,\zeta)d\zeta \right| \\
&& + \epsilon(\beta(t) - \alpha(t)) + \epsilon(\beta(t) - \alpha(t)).
\end{eqnarray*}
The last estimate above follows by \eqref{Eq_v2epsilon} and observation that if $u_x$ changes sign along a characteristic (as is the case in two last terms) then there exists $s$ such that $v(s) = 0$. 
Using the change of variables formula from Section \ref{Sec_ContPrel} we estimate further
\begin{eqnarray*}
I_1 &\le&\left| \int_{D_t^{unique,N}\cap [\alpha,\beta]\cap \{u_x(0,\cdot)<0\}} u_x^2(t,M_t(\zeta)M_t'(\zeta)d\zeta - \int_{D_t^{unique,N}\cap [\alpha,\beta]\cap \{u_x(0,\cdot)<0\}} u_x^2(0,\zeta)d\zeta \right| \\ &&+ 2 \epsilon(\beta(t)-\alpha(t)) \\
&\le& \int_{D_t^{unique,N}\cap [\alpha,\beta]\cap \{u_x(0,\cdot)<0\}} |u_x^2(t,M_t(\zeta))M_t'(\zeta) - u_x^2(0,\zeta)|d\zeta + 2 \epsilon(\beta(t)-\alpha(t))\\
&\le& \int_{D_t^{unique,N}\cap [\alpha,\beta]\cap \{u_x(0,\cdot)<0\}} |u_x^2(t,M_t(\zeta))||M_t'(\zeta)-1| d\zeta\\
 &&+ \int_{D_t^{unique,N}\cap [\alpha,\beta]\cap \{u_x(0,\cdot)<0\}} |u_x^2(t,M_t(\zeta)) - u_x^2(0,\zeta)| d\zeta +  2 \epsilon(\beta(t)-\alpha(t)) \\
&\le& \epsilon \int_{D_t^{unique,N}\cap [\alpha,\beta]\cap \{u_x(0,\cdot)<0\}} |u_x^2(t,M_t(\zeta))| d\zeta + \epsilon|\beta-\alpha| + 2 \epsilon(\beta(t)-\alpha(t))\\
&\le& 
\epsilon e^{t_{\epsilon}N} \int_{\mathbb{R}} |u_x^2(t,z)| dz + \epsilon|\beta-\alpha| + 2 \epsilon(\beta(t)-\alpha(t)),
\end{eqnarray*}
where in the last but one estimate we used \eqref{Eq_eNt}, \eqref{Eq_v2epsilon} and in the last estimate we used  the estimate $e^{-Nt} \le M_t'(\zeta)$  and the change of variables formula from Section \ref{Sec_ContPrel}, which lead to 
\begin{equation*}
e^{-Nt} \int_{D_t^{unique,N}\cap [\alpha,\beta] \cap \{u_x(0,\cdot)<0\} } u_x^2(t,M_t(\zeta))) d\zeta \le  \int_{M_t\left(D_t^{unique,N}\cap [\alpha,\beta]\cap  \{u_x(0,\cdot)<0\}\right)} u_x^2(t,z)dz. 
\end{equation*}

Hence,
$$I_1 \le
2 \epsilon \int_{\mathbb{R}} |u_x^2(t,z)| dz + \epsilon|\beta-\alpha| + 2 \epsilon(\beta(t)-\alpha(t)).$$

\noindent{\bf Term ${\bf I_2}$}
$$ 
I_2 \le \int_{[\alpha,\beta]\backslash I_T} u_x^2(0,\zeta)d\zeta < \epsilon
$$
by \eqref{Eq_es1}.

\noindent{\bf Term ${\bf I_3}$}

$$ 
I_3 \le \int_{(I_T \backslash D_t^{unique,N}) \cap [\alpha,\beta]} u_x^2(0,\zeta)d\zeta < \epsilon.
$$
by \eqref{Eq_es2}.

\noindent{\bf Term ${\bf I_5}$}

To begin, let us observe that by Lemma \ref{Lem_bwd} one can define characteristics starting from $t$ going backwards. In particular, for every $\gamma \in [\alpha(t),\beta(t)]$ there exists (in general nonunique) $\zeta_{\gamma}$ such that $\zeta_{\gamma}(t) = \gamma$.
We will show that  almost every $\gamma \in [\alpha(t),\beta(t)]$ satisfying $u_x(t,\gamma)<0$ belongs to $M_t(L_t^{unique})$. This implies directly that $I_5 = 0$. 

\begin{lemma}
\label{Lem_lem}
Let $u_x(t,\gamma)<0$ be a Lebesgue point of $u_x(t,\cdot)$ (difference quotients converge). Take $\kappa \neq \gamma$ so close to $\gamma$ that 
$(u(t,\kappa) - u(t,\gamma))/(\kappa - \gamma) < 0$. Let $\zeta(\cdot) := \zeta_{\gamma}(\cdot)$ and $\eta(\cdot):= \zeta_{\kappa}(\cdot)$. Then

\begin{equation*}
\forall_{s \in [0,t]} \omega(s) < \sqrt{LC}\tan(\sqrt{LC}t).
\end{equation*} 
In particular, $\omega(s)$ is bounded from above along characteristic $\zeta_{\gamma}(s)$ for $s \in [0,t]$ and $\zeta_{\gamma}(\cdot)$ does not branch.
\end{lemma}
\begin{proof}
By \eqref{Eq_2} $$\omega(t) \ge \sqrt{LC} \tan\left(-\sqrt{LC}(t-s) + \arctan\left(\frac {\omega(s)}{\sqrt{LC}}\right)\right).$$ Thus,
$$\arctan\left(\frac {\omega(t)}{\sqrt{LC}}\right) \ge -\sqrt{LC} (t-s) + \arctan\left(\frac {\omega(s)}{\sqrt{LC}}\right)$$ and
$$ \arctan\left(\frac {\omega(s)}{\sqrt{LC}}\right) \le \arctan\left(\frac {\omega(t)}{\sqrt{LC}}\right) + \sqrt{LC} t \le \sqrt{LC} t,$$
where we used $\omega(t)<0$. Hence,
$$\omega(s) \le \sqrt{LC} \tan(\sqrt{LC}t) \le \sqrt{LC},$$
the last inequality by \eqref{Eq_es111}.
\end{proof}

\begin{proposition}
If $u_x(t,\gamma)<0$ then $\zeta_{\gamma}$ is unique forwards and backwards for almost every $\gamma$. 
\end{proposition}
\begin{proof}
We present a sketch of the proof and refer the reader to \cite[Section 3]{TCGJ} for a rigorous proof in an analogous setting.
First observe that $\zeta_{\gamma} (\cdot)$ are unique forwards by Lemma \ref{Lem_lem}.
Secondly, if $\zeta^1 < \zeta^2$ and $\zeta^1(t)=\gamma = \zeta^2(t)$ then for every $\zeta^1 < \zeta < \zeta^2$ we have $\zeta(t)=\gamma$ by forwards uniqueness of characteristics. Thus for every $\gamma$ there exists a cone of characteristics $\zeta^{\lambda}_{\gamma}(\cdot)$ such that $\zeta^{\lambda}_{\gamma}(t) = \gamma$. If the cone of characteristics contains more than one characteristic, then it sweeps positive area, i.e. $|\bigcup_{\lambda}\bigcup_{s \in [0,t]} \{(s,\zeta_{\gamma^{\lambda}}(s))\}|>0$. Due to uniqueness forwards, cones for different $\gamma$ are disjoint. There can be at most countably many $\gamma$ with cones of positive area. Thus only for countably many $\gamma$ there exists more than one $\zeta_{\gamma}$ such that $\zeta_{\gamma}(t)=\gamma$.
\end{proof}
Consider now the solution $u^{bt}$ starting from $t$ backwards, see Lemma \ref{Lem_bwd}. Then $u^{bt}_x(0,\gamma)>0$ and $\omega^{bt}$ corresponding to the (nonunique) backward characteristic $\gamma^{bt}(\cdot)$, emanating from $\gamma$, is bounded from below by \eqref{Eq_2}. This implies that $\gamma^{bt}$ does not collide with any other backward characteristic, as in Figure \ref{Fig2} left. By Lemma \ref{Lem5} applied to the solution $u^{bt}$
 we obtain that for almost every such $\gamma$ the $\omega^{bt}$ corresponding to $\gamma^{bt}$ is bounded from \emph{above} along the characteristic for every $s \in [0,t]$ for almost every $\gamma$. This corresponds in the original solution to the fact that $\omega(s)$ is bounded from \emph{below} for almost every $\gamma$. 

We conclude that for almost every $\gamma$ the $\omega$ corresponding to $\zeta_{\gamma}$ is bounded from both above and below (for $\eta$ sufficiently close to $\zeta$). Hence, $\zeta_{\gamma} \in L_t^{unique}$. Consequently, $I_5=0$.

\noindent{\bf Term ${\bf I_4}$}

If $\zeta \in L_t^{unique} \backslash D_t^{unique,N}$ then there exists $\tau \in [0,t]$ such that $v(\tau)>N$ or $v(\tau)< -N$.
In the former case passing to the limit $\eta \to \zeta$ in \eqref{Eq_2} we obtain
\begin{equation}
\label{Eq_VTtau}
v(t) \ge \sqrt{LC} \tan \left( -\sqrt{LC} (t-\tau) + \arctan\left(\frac {v(\tau)}{\sqrt{LC}}\right)\right) > 0
\end{equation}
where the last inequality follows by assumption \eqref{Eq_NLC}. Consequently, such $\zeta$ gives no contribution to $I_4$.

In the latter case  
there exists $\tau \in [0,t]$ such that  $v(\tau)<-N$. Denote this set $Z$, i.e. $Z:= \{\zeta \in L_t^{unique} \cap [\alpha,\beta]: \exists_{\tau \in [0,t]} v(\tau) < -N \}$.
We will show that there exists $V_1>0$ such that $v(\sigma) < -V_1$ for all $\sigma \in [0,T]$.

To this end consider first $\sigma \in [0,\tau]$. Then, as in \eqref{Eq_VTtau}
\begin{eqnarray*}
v(\tau) &\ge& \sqrt{LC} \tan \left( -\sqrt{LC} (\tau-\sigma) + \arctan\left(\frac {v(\sigma)}{\sqrt{LC}}\right)\right) 
\end{eqnarray*}
and hence, extracting $v(\sigma)$ and using $v(\tau) \le -N$,
\begin{eqnarray*}
v(\sigma) &\le& \sqrt{LC} \tan \left( \sqrt{LC} (\tau- \sigma) + \arctan\left(\frac {-N}{\sqrt{LC}}\right)\right).
\end{eqnarray*}
By choice of $T$, see \eqref{Eq_es111} and $N$, see \eqref{Eq_NLC} we obtain
$$v(\sigma) \le -\sqrt{LC}.$$
On the other hand, for $\sigma \in [\tau,t]$ we obtain by equation \eqref{Eq_dotv}
\begin{equation}
\dot{v} = u^2 - \frac 1 2 v^2 - P \le \sup(u^2) + \sup(P).
\end{equation}
Hence, $v(\sigma) = v(\tau) + \int_{\sigma}^{\tau}\dot{v} \le -N + T(\sup(u^2) + \sup(P))$ and by \eqref{Eq_es12} we obtain $v(\sigma) \le - \sqrt{LC}$. 

To estimate $I_4$ set $V_1 := \sqrt{LC}$ and use Proposition \ref{Prop_doublebound} to obtain 

\begin{equation*}
e^{-\frac {2t}{V_1}((\sup u)^2 + \sup(P))} \le \frac {v^2(t) M_t'(\zeta)}{v^2(0) M_0'(\zeta)} \le e^{\frac {2t}{V_1}((\sup u)^2 + \sup(P))} 
\end{equation*}
This, in view of condition \eqref{Eq_es11}, gives
\begin{equation*}
e^{-\epsilon} \le \frac {v^2(t) M_t'(\zeta)}{v^2(0) M_0'(\zeta)}  \le e^{\epsilon},
\end{equation*}
which leads, for fixed $K>0$, to the estimate (we use the change of variables formula, see beginning of Section \ref{Sec_ContPrel}) 
\begin{eqnarray*}
&&\int_{M_t([\alpha,\beta] \cap L_t^{unique,K} \backslash D_t^{unique,N})} (u_x^-(t,z))^2 dz\\
&= &\int_{M_t(Z \cap L_t^{unique,K} \backslash D_t^{unique,N})} (u_x(t,z))^2 dz\\
&=& \int_{Z \cap L_t^{unique,K}\backslash D_t^{unique,N}} u_x^2(t,M_t(\zeta))dM_t(\zeta)\\
&=& \int_{Z \cap L_t^{unique,K}\backslash D_t^{unique,N}} u_x^2(t,M_t(\zeta))M_t'(\zeta) d\zeta \\
&\le& e^\epsilon \int_{Z \cap L_t^{unique,K}\backslash D_t^{unique,N}} u_x^2(0,\zeta) d\zeta.
\end{eqnarray*}

Passing to the limit $K \to \infty$ we obtain 
\begin{eqnarray*}
\int_{M_t([\alpha,\beta] \cap L_t^{unique} \backslash D_t^{unique,N})} (u_x^-(t,z))^2 dz &\le& e^{\epsilon} \int_{Z \cap L_t^{unique} \backslash D_t^{unique,N}} u_x^2(0,\zeta)d\zeta \\
&\le& e^{\epsilon} \int_{[\alpha,\beta] \backslash D_t^{unique,N}} u_x^2(0,\zeta)d\zeta \\
&\le& e^{\epsilon} \int_{[\alpha,\beta] \backslash I_T } u_x^2(0,\zeta)d\zeta +  e^{\epsilon} \int_{I_T \backslash D_t^{unique,N} } u_x^2(0,\zeta)d\zeta\\
&\le& 2\epsilon e^{\epsilon},
\end{eqnarray*}
where we used \eqref{Eq_es1} and \eqref{Eq_es2}.

Combining all the estimates, we obtain
\begin{eqnarray*}
&&\left| \int_{[\alpha(t),\beta(t)]} (u_x^-(t,z))^2 dz   - \int_{[\alpha,\beta]} (u_x^-(0,\zeta))^2 d\zeta \right| \\
&\le&{\bf I_1} + {\bf I_2} + {\bf I_3} + {\bf I_4} + {\bf I_5}\\
&\le& 2 \epsilon \int_{\mathbb{R}} |u_x^2(t,z)| dz + \epsilon|\beta-\alpha| + 2 \epsilon(\beta(t)-\alpha(t)) + \epsilon + \epsilon + 2\epsilon e^{\epsilon} + 0,
\end{eqnarray*}
which converges to $0$ as $\epsilon \to 0$ and $t \in [0,t_{\epsilon}]$. 

This proves Theorem \ref{Th_continuity} for $\alpha(t), \beta(t)$ -- leftmost characteristics. To obtain a proof for arbitrary characteristics, it is enough to show that for every $\zeta \in \mathbb{R}$ 
$$ \lim_{t \to 0} \int_{[\zeta^l(t),\zeta^r(t)]} (u_x^-(t,x))^2 dx  = 0,$$ where $\zeta^l$ and $\zeta^r$ are, respectively, the leftmost and rightmost characteristic. This, however, can be obtained by  passing $\beta \to \alpha^+$ in 
$$0= I_ 5 = \int_{[\alpha(t),\beta(t)] \backslash M_t(L_t^{unique})} (u_x^- (t,z))^2 dz.$$

\begin{remark}
\rm
One could ask whether it is possible to prove right continuity of $$\int_{[\alpha(t),\beta(t)]}u_x^+(t,z)^2dz$$ following the lines of proof of Theorem \ref{Th_continuity}. The answer is {\emph no} since even though the estimates for $I_1,I_2,I_3,I_4$ carry over, the term $I_5$ heavily depends on the fact that \emph{negative} part of derivative is considered. And indeed, $\int_{[\alpha(t),\beta(t)]}(u_x^+(t,z))^2dz$ is \emph{not} right continuous (see the peakon-antipeakon interaction in Introduction). Nevertheless, it is possible to prove existence of the limit $\lim_{t \to 0^+} \int_{[\alpha(t),\beta(t)] }u_x^+(t,x)^2dx$, see Section \ref{Sec_Th9}. 
\end{remark}

\section{Proof of Theorem \ref{Th_limit}}
\label{Sec_Th9}
\begin{lemma}
\label{Lem_limit}
Let $f: [0,\infty) \to (-\infty,0]$ be continuous, nonincreasing and such that $f(0)=0$. Let $g: [0,\infty) \to \mathbb{R}$ be bounded and satisfy $$g(t_2) - g(t_1) \ge f(t_2 - t_1)$$ for all $t_1,t_2$ such that  $0 \le t_1 < t_2$. Then there exists the limit
\begin{equation}
\label{Eq_glim}
\lim_{t \to 0^+} g(t).
\end{equation}
If, additionally, $f$ is Lipschitz continuous then $g$ has locally bounded variation.
\end{lemma}
\begin{proof}
Suppose the limit \eqref{Eq_glim} does not exist. Then there exists $\epsilon > 0$ such that  for every $\delta > 0$ there exist $0<t_1<t_2<\delta$ such that $|g(t_2) - g(t_1)| \ge \epsilon$. Let $h$ be so small that $f(h) \ge  -\epsilon/2$. Then for every $\delta < h$ there exist $0<t_1<t_2<\delta$ such that $g(t_2) - g(t_1) \ge \epsilon$, since $g(t_2)-g(t_1)\le -\epsilon$ is impossible due to the estimate $g(t_2) - g(t_1) \ge f(t_2-t_1) \ge f(\delta) \ge f(h) \ge -\epsilon/2$. Consequently, one can find a sequence $$h>t_1 > s_1 > t_2 > s_2 > \dots > t_n > s_n > \dots > 0$$ such that for $i=1,2, \dots$
\begin{eqnarray*}
g(t_i) - g(s_i) &\ge& \epsilon, \\
g(s_i) - g(t_{i+1}) &\ge& -\epsilon/2.
\end{eqnarray*}
This gives, for every $n>0$, the estimate
$$
g(t_1) - g(s_n) = \sum_{i=1}^{n} (g(t_i) - g(s_i))  +  \sum_{i=1}^{n-1} (g(s_i) - g(t_{i+1})) \ge n \epsilon + (n-1)(-\epsilon/2) \ge n \epsilon/2.
$$
Sending $n \to \infty$ we obtain a contradiction with boundedness of $g$ and thus existence of the limit \eqref{Eq_glim}.
If, additionally, $f$ is Lipschitz continuous, then there exists a constant $\kappa \ge 0$ such that $f(t_2 - t_1) \ge -\kappa (t_2 - t_1)$. Then function $\tilde{g}(t):= g(t)+ \kappa t$ is increasing and thus of locally bounded variation. Consequently, $g$ has locally bounded variation.
\end{proof}

To prove Theorem \ref{Th_limit} it is enough, by Lemma \ref{Lem_limit}, to show that for $t_2>t_1>0$
\begin{equation*}
\int_{[\alpha(t_2),\beta(t_2)]} (u_x^+(t_2,x))^2 dx \ge \int_{[\alpha(t_1),\beta(t_1)]} (u_x^+(t_1,x))^2 dx + f(t_2 - t_1) 
\end{equation*}
for some function $f$ satisfying assumptions of Lemma \ref{Lem_limit}.
To this end, fix $T>0$, $V_4 := 4 (\sup(u)^2 + \sup(P))$ and 
find $\Delta t_{\epsilon} < \max(1, \pi /4)$ so small that 
\begin{eqnarray}
\sqrt{LC} \tan (-\sqrt{LC} \Delta t_{\epsilon} + \arctan(V_4/\sqrt{LC})) &\ge& V_4 /2, \label{Eq_kolko1}
\end{eqnarray}
Take now any $t_1,t_2$ such that $0< t_1 < t_2 \le \Delta t_{\epsilon}$. Before we proceed with the estimate, we
\begin{itemize}
\item Define $M_{t_1t_2}(\gamma) := \gamma^l (t_2)$ where $\gamma^l$ is the leftmost characteristic satisfying $\gamma(t_1)=\gamma$. Function $M_{t_1t_2}$ is an analog, for initial time $t_1$, of function $M_t$ defined in Definition \ref{Def_Mt}.
\item Define $L_{t_1 t_2}^{unique} := [L_{t_2 - t_1}^{unique}$ for shifted solution $u(\cdot -t_1,\cdot)$], where $L_t^{unique}$ is defined in Definition \ref{Def_Lt}. Thus $L_{t_1,t_2}^{unique}$ is the $L_{t_2 - t_1}^{unique}$ with initial condition at time $t_1$.
\item Observe that due to uniqueness we have $$M_{t_1t_2} (L_{t_1t_2}^{unique} \cap [\alpha(t_1),\beta(t_1)]) \subset [\alpha(t_2),\beta(t_2)]$$ even though not necessarily $M_{t_1t_2}([\alpha(t_1),\beta(t_1)]) \subset [\alpha(t_2),\beta(t_2)]$.
\item Observe that by \eqref{Eq_2} combined with \eqref{Eq_kolko1}, if $u_x(t_1,\gamma) \ge V_4$ then $u_x(t_2,M_{t_1t_2}(\gamma)) \ge V_4/2$ and, consequently, by \eqref{Eq_BigN}
\begin{equation}
\frac {v^2(t_2) M_{t_1 t_2}'(\gamma)}{v^2(t_1)}  \ge e^{-\frac {2(t_2 - t_1)}{V_4/ 2}((\sup u)^2 + \sup(P))} = e^{- (t_2 - t_1)} \ge 1 - (t_2 - t_1). \label{Eq_kwadrat}
\end{equation}
\item Observe that if $V_4 \ge v(t_1) \ge 0$ then  for every $t \in [t_1,t_2]$ we have, by \eqref{Eq_2},
\begin{eqnarray*}
v(t) &\ge& \sqrt{LC} \tan(-\sqrt{LC} (t_2 - t_1)) \ge - \sqrt{LC}, \\
v(t) &\le& V_4 +  (\Delta t_\epsilon) \sup(u^2) \le V_4 + \sup(u^2).
\end{eqnarray*}
and hence $$\sup_{[t_1,t_2]} |v(t)| \le \tilde{V} := \max(\sqrt{LC}, V_4 + \sup(u^2)).$$
\end{itemize}

Using these observations and the change of variables described at the beginning of Section \ref{Sec_ContPrel} we obtain, for fixed $V_3 > V_4$,

\begin{eqnarray*}
\int_{[\alpha(t_2),\beta(t_2)]} (u_x^+ (t_2,z))^2 dz &\ge& \int_{M_{t_1 t_2} (L_{t_1 t_2}^{unique} \cap \{V_3 \ge u_x(t_1,\cdot) \ge V_4\} \cap [\alpha(t_1),\beta(t_1)])} (u_x (t_2,z))^2  dz\\
&& + \int_{M_{t_1 t_2} (L_{t_1 t_2}^{unique} \cap \{V_4 > u_x(t_1,\cdot) \ge 0\} \cap [\alpha(t_1),\beta(t_1)])} (u_x (t_2,z))^2  dz\\  &=& T_1 + T_2.
\end{eqnarray*}

\begin{eqnarray*}
T_1 &=& \int_{L_{t_1 t_2}^{unique} \cap \{V_3 \ge u_x(t_1,\cdot) \ge V_4\} \cap [\alpha(t_1),\beta(t_1)]} u_x (t_2,M_{t_1t_2}(\gamma))^2  dM_{t_1t_2}(\gamma)\\
&=& \int_{L_{t_1 t_2}^{unique} \cap \{V_3 \ge u_x(t_1,\cdot) \ge V_4\} \cap [\alpha(t_1),\beta(t_1)]} u_x (t_2,M_{t_1t_2}(\gamma))^2  M_{t_1t_2}'(\gamma) d\gamma \\
&\ge& \left(1- (t_2 - t_1)\right) \int_{L_{t_1 t_2}^{unique} \cap \{V_3 \ge u_x(t_1,\cdot) \ge V_4\} \cap [\alpha(t_1),\beta(t_1)]} u_x (t_1,\gamma)^2   d\gamma,\\
&&\mbox{the last inequality by  \eqref{Eq_kwadrat}},\\
T_2 &=& \int_{L_{t_1 t_2}^{unique} \cap \{V_4 > u_x(t_1,\cdot) \ge 0\} \cap [\alpha(t_1),\beta(t_1)]} u_x (t_2,M_{t_1t_2}(\gamma))^2  dM_{t_1t_2}(\gamma)\\
&=& \int_{L_{t_1 t_2}^{unique} \cap \{V_4 \ge u_x(t_1,\cdot) \ge 0\} \cap [\alpha(t_1),\beta(t_1)]} u_x (t_2,M_{t_1t_2}(\gamma))^2  M_{t_1t_2}'(\gamma) d\gamma \\
&\ge& \int_{L_{t_1 t_2}^{unique} \cap \{V_4 \ge u_x(t_1,\cdot) \ge 0\} \cap [\alpha(t_1),\beta(t_1)]} u_x (t_1,\gamma)^2 d\gamma   \\
&& - (t_2 - t_1) \int_{L_{t_1 t_2}^{unique} \cap \{V_4 \ge u_x(t_1,\cdot) \ge 0\} \cap [\alpha(t_1),\beta(t_1)]}  \sup_{t \in [t_1,t_2]} \left| \frac {d}{dt} (v^2 M_{t_1t} ') \right| d\gamma \\
&\ge& \int_{L_{t_1 t_2}^{unique} \cap \{V_4 \ge u_x(t_1,\cdot) \ge 0\} \cap [\alpha(t_1),\beta(t_1)]} u_x (t_1,\gamma)^2 d\gamma   \\
&& - (t_2 - t_1) \int_{L_{t_1 t_2}^{unique} \cap \{V_4 \ge u_x(t_1,\cdot) \ge 0\} \cap [\alpha(t_1),\beta(t_1)]} \sup_{t \in [t_1,t_2]} |2v(u^2 - P) M_{t_1 t}'|   d\gamma  \\
&\ge& \int_{L_{t_1 t_2}^{unique} \cap \{V_4 \ge u_x(t_1,\cdot) \ge 0\} \cap [\alpha(t_1),\beta(t_1)]} u_x (t_1,\gamma)^2 d\gamma   \\
&& - 2 \tilde{V} (\sup(u^2) + \sup(P)) e^{\Delta t_{\epsilon} \tilde{V}} (t_2 - t_1)  \int_{L_{t_1 t_2}^{unique} \cap \{V_4 \ge u_x(t_1,\cdot) \ge 0\} \cap [\alpha(t_1),\beta(t_1)]}  d\gamma  
\end{eqnarray*}

Thus, 
\begin{eqnarray*}
\int_{[\alpha(t_2),\beta(t_2)]} (u_x^+ (t_2,z))^2 dz &\ge& T_1 + T_2\\
 &\ge& \int_{L_{t_1 t_2}^{unique} \cap \{V_3 \ge u_x(t_1,\cdot) \ge 0\} \cap [\alpha(t_1),\beta(t_1)]} u_x (t_1,\gamma)^2   d\gamma \\
   &&- (t_2 - t_1) (2 \tilde{V} (\sup(u^2) + \sup(P)) e^{\Delta t_{\epsilon} \tilde{V}} ) (\beta(t_1) - \alpha(t_1))  \\
  && - (t_2 - t_1) \int_{[\alpha(t_1),\beta(t_1)]} u_x (t_1,\gamma)^2   d\gamma.
\end{eqnarray*}

Passing to the limit $V_3 \to \infty$ we end up with:
\begin{eqnarray}
\label{Eq_minus2eps}
\int_{[\alpha(t_2),\beta(t_2)]} u_x^+ (t_2,z)^2 dz &\ge& \int_{[\alpha(t_1),\beta(t_1)] }  u_x^+(t_1,\gamma)^2 d\gamma \\
   &&- (t_2 - t_1) (2 \tilde{V} (\sup(u^2) + \sup(P)) e^{\Delta t_{\epsilon} \tilde{V}} ) (\beta(t_1) - \alpha(t_1))   \nonumber\\
  && - (t_2 - t_1) \int_{[\alpha(t_1),\beta(t_1)]} u_x (t_1,\gamma)^2   d\gamma,\nonumber
\end{eqnarray}
where we used the identity $\int_{L_{t_1t_2}^{unique} \cap [\alpha(t_1),\beta(t_1)]}  u_x^+(t_1,\gamma)^2 d\gamma = \int_{[\alpha(t_1),\beta(t_1)]}  u_x^+(t_1,\gamma)^2 d\gamma$, which follows by the same argumentation as the proof of $I_5 = 0$ in Section \ref{Sec_ProofCont}.

Using Lemma \ref{Lem_limit} we conclude the proof of Theorem \ref{Th_limit}. 

\section{Proof of Theorem \ref{Th_general}}
   \label{Sec_ProofThGeneral}
Properties \eqref{EqTh81} and \eqref{EqTh83} as well as $BV$ regularity of  $t \mapsto \int_{[\alpha(t),\beta(t)]} (u_x^+(t,x))^2dx$ are simple consequences of Theorems \ref{Th_continuity} and  \ref{Th_limit}, respectively, applied to function $\tilde{u}(t,x) := u(t-t_0,x)$. To prove \eqref{EqTh82} and \eqref{EqTh84} as well as $BV$ regularity of $t \mapsto \int_{[\alpha(t),\beta(t)]} (u_x^-(t,x))^2dx$, it is enough to apply Theorems \ref{Th_continuity} and \ref{Th_limit} to function $\tilde{u} (t,x) := -u(t_0 - t, x)$, which is a solution of the Camassa-Holm equation by Lemma \ref{Lem_bwd}. Indeed, to prove \eqref{EqTh82}, we denote $\tilde{\alpha}(t) := \alpha(t_0-t), \tilde{\beta}(t) := \beta(t_0-t)$ and obtain, by Theorem \ref{Th_limit}:
\begin{eqnarray*}
\lim_{t \to 0^+} \int_{[\tilde{\alpha}(t),\tilde{\beta}(t)]} (\tilde{u}_x^-(t,x))^2 dx = \int_{[\alpha(t_0),\beta(t_0)]} (\tilde{u}_x^-(0,x))^2 dx,
\end{eqnarray*}
which is equivalent to
\begin{eqnarray*}
\lim_{t \to 0^+} \int_{[\tilde{\alpha}(t),\tilde{\beta}(t)]} ((-{u}_x)^-(t_0-t,x))^2 dx = \int_{[\alpha(t_0),\beta(t_0)]} ((-{u}_x)^-(t_0,x))^2 dx
\end{eqnarray*}
and 
\begin{eqnarray*}
\lim_{t \to 0^+} \int_{[\tilde{\alpha}(t),\tilde{\beta}(t)]} ({u}_x^+(t_0-t,x))^2 dx = \int_{[\alpha(t_0),\beta(t_0)]} ({u}_x^+(t_0,x))^2 dx.
\end{eqnarray*}
Changing variables $t \to t_0 - t$ we conclude. Similarly, to prove $BV$ regularity of $t \mapsto \int_{[\alpha(t),\beta(t)]} (u_x^-(t,x))^2 dx$ we observe that by Theorem \ref{Th_limit}
function $$t \mapsto \int_{[\tilde{\alpha}(t),\tilde{\beta}(t)]} (\tilde{u}_x^+(t,x))^2 dx = \int_{[\tilde{\alpha}(t),\tilde{\beta}(t)]} (u_x^-(t_0-t,x))^2 dx$$ is locally of bounded variation.
Change of the time variable $t \to t_0 - t$ concludes.

\section{Proof of Theorem \ref{Th_cadlag}}
\label{Sec_Th10}
Fix a smooth compactly supported function $\phi: \mathbb{R} \to [0,\infty)$. We need to prove that, for fixed time $t_0$,
\begin{eqnarray*}
\lim_{t \to t_0^+} \int_{\mathbb{R}} \phi(x)(u_x^-(t,x))^2 dx &=& \int_{\mathbb{R}} \phi(x)(u_x^-(t_0,x))^2 dx,\\
\lim_{t \to t_0^-} \int_{\mathbb{R}} \phi(x) (u_x^+(t,x))^2 dx &=& \int_{\mathbb{R}} \phi(x)(u_x^+(t_0,x))^2 dx,
\end{eqnarray*}
existence and linearity (with respect to $\phi$) of limits
\begin{eqnarray*} 
\lim_{t \to t_0^+} \int_{\mathbb{R}} \phi(x)(u_x^+(t,x))^2 dx,\\
\lim_{t \to t_0^-} \int_{\mathbb{R}} \phi(x)(u_x^-(t,x))^2 dx,
\end{eqnarray*}
as well as that 
\begin{eqnarray*}
t &\mapsto& \int_{\mathbb{R}} \phi(x)(u_x^-(t,x))^2 dx, \\
t &\mapsto& \int_{\mathbb{R}} \phi(x)(u_x^+(t,x))^2 dx
\end{eqnarray*}
 have locally bounded variation on $[0,\infty)$.

As a matter of fact, however, it is enough to prove the $BV_{loc}$ property of $$t \mapsto \int_{\mathbb{R}} \phi(x)(u_x^+(t,x))^2 dx,$$ which guarantees existence of both left and right limits and to determine $$\lim_{t \to t_0^+} \int_{\mathbb{R}} \phi(x)(u_x^-(t,x))^2 dx.$$ The other properties follow then by duality.

Let us first consider the limit ${\lim_{t \to t_0^+} \int_{\mathbb{R}} \phi(x)(u_x^-(t,x))^2 dx}$. Fix $\epsilon > 0$ and approximate $\phi$ by
\begin{equation*}
\phi^{\epsilon}:= \sum_{i=0}^{I-1} c_i\bold{1}_{[a_i,a_{i+1})}
\end{equation*}
in such a way that $a_0<a_1 < \dots <a_I$, $c_i > 0$, $c_0<\epsilon$, $c_{I-1}<\epsilon$, $|c_i-c_{i+1}|<\epsilon$ and $$\sup_{x \in \mathbb{R}} |\phi^{\epsilon}(x) - \phi(x)| \le \epsilon.$$
Define for $t>t_0$
\begin{equation*}
\phi^{\epsilon} (t,x) := \sum_{i=0}^{I-1} c_i\bold{1}_{[a_i(t),a_{i+1}(t))} (x),
\end{equation*}
where $a_i(t)$ are the leftmost characteristics emanating from $a_i$ at time $t_0$. By Theorem \ref{Th_general} as well as finite speed of propagation of characteristics, we can find $\delta>0$ such that if $t_0 \le t \le t_0+\delta$  then 
\begin{itemize}
\item $$\sup_{x \in \mathbb{R}} |\phi^{\epsilon}(t,x) - \phi^{\epsilon}(x)| < \epsilon,$$
\item $$\left| \int_{[a_i(t),a_{i+1}(t)]} u_x^-(t,x)^2 dx - \int_{[a_i(t_0),a_{i+1}(t_0)]} u_x^-(t_0,x)^2 dx \right| < {\epsilon}\slash\sum_{j=0}^{I-1}c_j,$$
\end{itemize}
for $i=0,1,\dots,I-1$.
Using these estimates we calculate
\begin{eqnarray*}
\left| \int_{\mathbb{R}} \phi(x) (u_x^-(t,x))^2 dx -  \int_{\mathbb{R}} \phi(x) (u_x^-(t_0,x))^2 dx \right| \le \\
\int_{\mathbb{R}} |\phi(x) - \phi^\epsilon(t,x)| (u_x^-(t,x))^2 dx + \left| \int_{\mathbb{R}} \phi^{\epsilon}(t,x) u_x^-(t,x)^2 dx - \int_{\mathbb{R}} \phi^\epsilon(x)u_x^-(t_0,x)^2 dx  \right|\\ + \int_{\mathbb{R}} |\phi^\epsilon(x) - \phi(x)| u_x^-(t_0,x)^2 dx = I_1 + I_2 + I_3.
\end{eqnarray*}
Now,
$$I_1 = \int_{\mathbb{R}} |\phi(x) - \phi^\epsilon(t,x)| (u_x^-(t,x))^2 dx \le 2\epsilon \sup_{t} \int_{\mathbb{R}} u_x^2(t,x)dx, $$
\begin{eqnarray*}
I_2 &=&  \left| \sum_{i=0}^{I-1} c_i \left(\int_{[a_i(t),a_{i+1}(t))} u_x^-(t,x)^2 dx - \int_{[a_i(t_0),a_{i+1}(t_0))} u_x^-(t_0,x)^2 dx \right) \right|\\
&\le& \sum_{i=0}^{I-1} c_i \left|\int_{[a_i(t),a_{i+1}(t))} u_x^-(t,x)^2 dx - \int_{[a_i(t_0),a_{i+1}(t_0))} u_x^-(t_0,x)^2 dx \right| \\
&=& \sum_{i=0}^{I-1} c_i \left|\int_{[a_i(t),a_{i+1}(t)]} u_x^-(t,x)^2 dx - \int_{[a_i(t_0),a_{i+1}(t_0)]} u_x^-(t_0,x)^2 dx \right| \le \epsilon
\end{eqnarray*}
and
$$I_3 \le \epsilon \sup_{t} \int_{\mathbb{R}} u_x^2(t,x)dx. $$
We conclude that 
\begin{equation*}
\lim_{t \to t_0^+} \left| \int_{\mathbb{R}} \phi(x) (u_x^-(t,x))^2 dx -  \int_{\mathbb{R}} \phi(x) (u_x^-(0,x))^2 dx \right|  = 0.
\end{equation*}

\noindent Next, to prove that  $$t \mapsto \int_{\mathbb{R}} \phi(x)(u_x^-(t,x))^2 dx$$  has locally bounded variation we recall that, by \eqref{Eq_minus2eps}, 
\begin{eqnarray}
\int_{[\alpha(t_2),\beta(t_2)]} u_x^+ (t_2,z)^2 dz &\ge& \int_{[\alpha(t_1),\beta(t_1)] }  u_x^+(t_1,\gamma)^2 d\gamma \label{Eq46} \\
   &&- K (t_2 - t_1) (\beta(t_1) - \alpha(t_1))   \nonumber\\
  && - (t_2 - t_1) \int_{[\alpha(t_1),\beta(t_1)]} u_x (t_1,\gamma)^2   d\gamma.\nonumber
\end{eqnarray}
Take any bounded compactly supported test function $\phi: \mathbb{R} \to [0,\infty)$.
Fix $\epsilon > 0$ and approximate $\phi$, as before, by
\begin{equation*}
\phi^{\epsilon}:= \sum_{i=0}^{I-1} c_i\bold{1}_{[a_i,a_{i+1})}
\end{equation*}
in such a way that $a_0<a_1 < \dots <a_I$, $c_i > 0$ and $$\sup_{x \in \mathbb{R}} |\phi^{\epsilon}(x) - \phi(x)| \le \epsilon.$$
Let $a_i(t)$ be the leftmost characteristics emanating from $a_i$. 
Then, by \eqref{Eq46}
\begin{eqnarray*}
\sum_{i=0}^{I-1} c_i \int_{[a_i(t_2),a_{i+1}(t_2)]} u_x^+(t_2,z)^2 dz &\ge& \sum_{i=0}^{I-1} c_i \int_{[a_i(t_1),a_{i+1}(t_1)]} u_x^+ (t_1,\gamma)^2 d\gamma  \\
 && -(t_2 - t_1) \left(K(a_I(t_1) - a_0(t_1)) + \int_{[a_0(t_1),a_I(t_1)]} u_x(t_1,\gamma)^2 d\gamma \right) \\
 &\ge& \int_{\mathbb{R}} \phi u_x^+(t_1,\gamma)^2 d\gamma  - \epsilon \int_{\mathbb{R}} u_x^2 (t_1,\gamma) d \gamma  \\
 && -(t_2 - t_1) \left(K(a_I(t_1) - a_0(t_1)) + \int_{[a_0(t_1),a_I(t_1)]} u_x(t_1,\gamma)^2 d\gamma \right).\\
\end{eqnarray*}
On the other hand,
\begin{eqnarray*}
\sum_{i=0}^{I-1} c_i \int_{[a_i(t_2),a_{i+1}(t_2)]} u_x^+(t_2,z)^2 dz &\le& \int_{\mathbb{R}} \phi u_x^+ (t_2,z)^2 dz  +  \int_{\mathbb{R}} \left|\phi - \sum_{i=0}^{I-1} c_i \bold{1}_{[a_i(t_2), a_{i+1}(t_2))} (z)  \right| u_x^2(t_2,z)dz.
\end{eqnarray*}
Now, denoting $M_{t_1t_2} \phi := \phi(M_{t_1 t_2} ^{-1})$ the pushforward of the function $\phi$ with $M_{t_1t_2}$, where $M_{t_1 t_2} ^{-1}$ is the right-continuous inverse, we obtain
\begin{eqnarray*}
\left|\phi - \sum_{i=0}^{I-1} c_i \bold{1}_{[a_i(t_2), a_{i+1}(t_2))}  \right| &\le& |\phi - M_{t_1t_2} \phi| + \left|M_{t_1t_2} \phi - \sum_{i=0}^{I-1} c_i \bold{1}_{[a_i(t_2), a_{i+1}(t_2))}\right| \\
&\le&  \sup_x \sup_{|\xi| \le \sup(u)(t_2-t_1)}|\phi(x) - \phi(x+ \xi)| + \left|M_{t_1t_2} \left(\phi - \sum_{i=0}^{I-1} c_i \bold{1}_{[a_i(t_1), a_{i+1}(t_1))}\right)\right| \\
&\le& Lip(\phi) \sup(u) (t_2 - t_1)  + \epsilon.
\end{eqnarray*}
Combining these estimates we obtain 
\begin{eqnarray*}
\int_{\mathbb{R}} \phi u_x^+ (t_2,z)^2 dz  &\ge&  \int_{\mathbb{R}} \phi u_x^+(t_1,\gamma)^2 d\gamma  - \epsilon \int_{\mathbb{R}} u_x^2 (t_1,\gamma) d \gamma  \\
 && -(t_2 - t_1) \left(K(a_I(t_1) - a_0(t_1)) + \int_{[a_0(t_1),a_I(t_1)]} u_x(t_1,\gamma)^2 d\gamma \right)  \\
 && - (Lip(\phi) \sup(u) (t_2 - t_1)  + \epsilon ) \int_{{\rm supp}(\phi) + [-(t_2-t_1)\sup(u),(t_2-t_1)\sup(u)]} u_x^2(t_2,z)dz.
 \end{eqnarray*}
Passing to the limit $\epsilon \to 0$ we are left with
\begin{eqnarray*}
\int_{\mathbb{R}} \phi u_x^+ (t_2,z)^2 dz  &\ge&  \int_{\mathbb{R}} \phi u_x^+(t_1,\gamma)^2 d\gamma\\
 && -(t_2 - t_1) \left(K(a_I(t_1) - a_0(t_1)) + \int_{[a_0(t_1),a_I(t_1)]} u_x(t_1,\gamma)^2 d\gamma \right)  \\
 && - (Lip(\phi) \sup(u) (t_2 - t_1)   ) \int_{\mathbb{R}} u_x^2(t_2,z)dz,
 \end{eqnarray*}
 where $K$ is defined in Section \ref{Sec_Th9}, see inequalities \eqref{Eq_minus2eps} and \eqref{Eq46}. 
This allows us to use Lemma \ref{Lem_limit} to conclude that function 
$t \mapsto \int_{\mathbb{R}} \phi u_x^+(t,z)^2 dz$ has locally bounded variation and, as a consequence, the existence of its right limits.

\end{document}